\documentclass[11pt]{article}

\usepackage{cite}

\usepackage[english]{babel}
\usepackage{amsmath}
\usepackage{amsthm}
\usepackage{amssymb}
\usepackage{tikz}
\usepackage{tikz-cd}
\usepackage{csquotes}

\usepackage{hyperref}

\usepackage{geometry}
\newtheorem{thm}{Theorem}[section]
\newtheorem*{thm*}{Theorem}
\newtheorem*{mainthm}{Main Theorem}
\newtheorem{theorem}{Theorem}

\newtheorem{corollary}[theorem]{Corollary}
\newtheorem{fact}[thm]{Fact}

\newtheorem{lem}[thm]{Lemma}
\newtheorem{cor}[thm]{Corollary}

\theoremstyle{definition}
\newtheorem{ex}[thm]{Example}
\newtheorem{defn}[thm]{Definition}

\theoremstyle{remark}
\newtheorem{rem}[thm]{\textsc{Remark}}

\newtheorem{notation}[thm]{\textsc{Notation}}

\newcommand{\IN}{\mathbb{N}}
\newcommand{\IZ}{\mathbb{Z}}

\newcommand{\Csp}{\op{Csp}}
\newcommand{\hCsp}{\mi{h}\!\op{Csp}}
\newcommand{\rCsp}{\op{Csp}^{\mathrm{red}}}

\newcommand{\iCsp}{\op{Csp}^{\mathrm{inj}}}
\newcommand{\iiCsp}{\op{Csp}^{2\mathrm{inj}}}
\newcommand{\riCsp}{\op{Csp}^{\mathrm{red},\mathrm{inj}}}
\newcommand{\riiCsp}{\op{Csp}^{\mathrm{red},2\mathrm{inj}}}

\newcommand{\op}[1]{\operatorname{#1}}
\newcommand{\id}{\op{id}}
\newcommand{\Id}{{\op{Id}}}
\newcommand{\colim}{\op{colim}}
\newcommand{\Map}{\op{Map}}

\newcommand{\Aut}{\op{Aut}}

\newcommand{\Fin}{\mi{Fin}}

\newcommand{\ga}{\alpha}

\newcommand{\gd}{\delta}
\newcommand{\gD}{\Delta}
\newcommand{\gc}{\gamma}

\newcommand{\gL}{\Lambda}
\newcommand{\gs}{\sigma}
\newcommand{\gS}{\Sigma}

\newcommand{\gO}{\Omega}

\newcommand{\mi}[1]{\mathit{#1}}
\newcommand{\mrm}[1]{\mathrm{#1}}
\newcommand{\hofib}{\mathrm{hofib}}

\newcommand{\mc}[1]{\mathcal{#1}}

\newcommand{\mcB}{\mc{B}}
\newcommand{\mcC}{\mc{C}}
\newcommand{\mcD}{\mc{D}}
\newcommand{\mcE}{\mc{E}}
\newcommand{\mcF}{\mc{F}}

\newcommand{\cl}[1]{\overline{#1}}
\newcommand{\ul}[1]{\underline{#1}}
\newcommand{\wt}[1]{\widetilde{#1}}

\newcommand{\cd}{\bullet}
\renewcommand{\c}[1]{\mathtt{#1}}
\newcommand{\Cob}{\c{Cob}}
\newcommand{\qand}{\quad \text{and} \quad}
\newcommand{\blank}{\underline{\ \ }}

\title{
Locally (co)Cartesian fibrations as realisation fibrations \\ and the classifying space of cospans
}

\author{Jan Steinebrunner}
\date{August 17, 2021}

\begin{document}

\maketitle
\renewcommand{\thefootnote}{\fnsymbol{footnote}} 
\renewcommand{\thefootnote}{\arabic{footnote}} 
\begin{abstract}
    We show that the conditions in Steimle's 
    ``additivity theorem for cobordism categories"
    can be weakened to only require \emph{locally} (co)Cartesian fibrations,
    making it applicable to a larger class of functors.
    As an application we compute the difference in classifying spaces
    between the infinity category of cospans of finite sets 
    and its homotopy category.
\end{abstract}


\section{Introduction}
When trying to compute the homotopy type of the classifying space $B\mcC$ of some category $\mcC$
it can be an effective strategy to
first construct a functor $F:\mcC \to \mcD$ to another category
and to then determine the homotopy fiber of the induced map on classifying spaces.
Quillen's \cite[Theorem~B]{Qui73} states that, under certain conditions, the homotopy fiber of $BF: B\mcC \to B\mcD$
is computed by the classifying space of the comma category $(F\downarrow d)$ of the functor $F$.
Here $(F\downarrow d)$ can be thought of as a categorical 
analogue of the path fiber in topology.  
However, there also is a much simpler variant of the fiber of a functor:
\begin{defn}
    The \emph{genuine fiber} $P_{|b} = \{b\}\times_\mcB \mcE$ of a functor $P:\mcE \to \mcB$ 
    at an object $b \in \mcB$
    has as objects those $e \in \mcE$ with $P(e) = b$ 
    and as morphisms those $(f:e \to e') \in \mcE$ with $P(f) = \id_b$.
\end{defn}
We will study functors for which 
already the \emph{genuine} fiber computes the homotopy fiber.
\begin{defn}[\!\!\cite{Rez14}]
    A functor $P:\mcE \to \mcB$ is a \emph{realisation fibration}
    if, for any other functor $F:\mcC \to \mcB$,
    the following diagram is homotopy Cartesian:
    \[
        \begin{tikzcd}
            B(\mcC \times_{\mcB} \mcE) \ar[r] \ar[d] & B\mcE \ar[d, "BP"] \\
            B\mcC \ar[r, "BF"] & B\mcB .
        \end{tikzcd}
    \]
\end{defn}

Intuitively, this property states that $BP$ behaves like a fibration:
pullbacks along $P$ become homotopy pullbacks along $BP$.
While this does not mean that $BP$ is actually a fibration,
it still implies that for any $b \in \mcB$ the genuine fiber $B\mcE_{|b}$ 
is the homotopy fiber of $B\mcE \to B\mcB$ at $b$.

In \cite{St18} Steimle proves an ``additivity theorem for bordism categories'',
which provides us with a supply of realisation fibrations by giving categorical criteria 
for a functor to be a realisation fibration.
We generalise this theorem by showing that the criteria can be weakened.
In the special case of discrete categories our main theorem is:
\begin{mainthm}
    If a functor is locally Cartesian and locally coCartesian, then it is a realisation fibration.
\end{mainthm}

Whether a functor $P:\mcE \to \mcB$ is locally (co)Cartesian can be checked as follows:
\begin{defn}\label{defn:locCart}
    A morphism $f:x \to y$ is called \emph{locally $P$-Cartesian} if for all $x' \in \mcE$
    with $P(x') = P(x)$ post-composition with $f$ induces a bijection
    \[
        (f \circ \blank): \mcE_{|P(y)}(x', x) = \{ g:x' \to x \;|\; P(g) = \id_{P(x)}\} 
        \longrightarrow  \{ h:x' \to y \;|\; P(h) = P(f)\}.
    \]
    The functor $P$ is called \emph{locally Cartesian} if for all morphisms $(k:a \to b) \in \mcB$
    and all lifts $y \in \mcE$ with $P(y)=b$ there is $x \in \mcE$ 
    and a locally $P$-Cartesian morphism $f:x \to y$ with $P(f) = k$.
    
    We say that $P$ is \emph{locally coCartesian} if $P^{op}: \mcE^{op} \to \mcB^{op}$ is
    locally Cartesian.
\end{defn}

With future applications in mind, we will prove our main theorem in the more general context of 
\emph{weakly unital topological categories}. It reads as follows:
\begin{theorem}[See \ref{thm:wu-locadd}]\label{thm:main}
    Let $P: \mcE \to \mcB$ be a weakly unital functor of weakly unital topological categories
    such that $\mcB$ is fibrant, $P$ is a local fibration, $P$ is locally Cartesian,
    and $P$ is locally coCartesian. Then $P$ is a realisation fibration.
\end{theorem}

This generalises \cite[Theorem 2.3]{St18}, where $P$ was required to be Cartesian and coCartesian.
Our proof will follow the same structure as in Steimle's paper.
In particular, we will also prove intermediate versions of the main theorem 
in the context of quasicategories and semi-Segal spaces
generalising \cite[Theorem 2.11 and 2.14]{St18}, respectively.
These theorems are stated as \ref{thm:qCatAdd} and \ref{thm:sS-localAdd} in the main text 
and might be of independent interest.

\ 

As an application of Theorem \ref{thm:main} we study the $\infty$-category $\Csp$ of cospans of finite sets.
The objects of this category are finite sets and the morphisms from $A$ to $B$
are cospans, that is diagrams of the shape $(A \to W \leftarrow B)$.
The $2$-morphisms $\ga: (A \to W \leftarrow B) \Rightarrow (A \to V \leftarrow B)$ 
are compatible bijections $\ga:W \to V$.
As we recall in \autoref{sec:Cosp} this defines an $(\infty,1)$- or $(2,1)$-category.
We write $\hCsp$ for the homotopy category, where two cospans are identified
whenever there is a bijection between them.

Theorem \ref{thm:main} does not directly apply to the quotient functor $P:\Csp \to \hCsp$.
We therefore construct a further quotient $R:\hCsp \to \rCsp$.
In this ``reduced cospan'' category $\rCsp$ two cospans $[A \to W_i \leftarrow B]$
are identified if they become isomorphic after removing those points of $W_i$
that do not lie in the image of $A$ or $B$.
We will see that a morphism $[A\to W \leftarrow B]$ in $\hCsp$ is locally $R$-Cartesian if and only if
it is \emph{reduced}, i.e.\ if $A \amalg B \to W$ is surjective.
This implies that the two functors $R$ and $R \circ P$ are locally (co)Cartesian,
but not (co)Cartesian. 
Theorem \ref{thm:main} therefore applies to $R$ and $R\circ P$, 
while the main theorem of \cite{St18} does not,
and yields two homotopy fiber sequences:

\begin{theorem}\label{thm:Cosp-fib}
    There is a commutative diagram of homotopy fiber sequences:
    \[
        \begin{tikzcd}
            Q(S^1) \ar[r] \ar[d] & B\Csp \ar[d, "P"] \ar[r, "R \circ P"] 
                                 & B\rCsp \ar[d, equal]\\
            S^1 \ar[r] & B\hCsp \ar[r, "R"] & B\rCsp 
        \end{tikzcd}
    \]
    where $Q(S^1) := \colim_{n \to \infty} \gO^n S^{n+1}$ is the free infinite loop space on $S^1$.
    As a consequence the homotopy fiber of $P:B\Csp \to B\hCsp$ is equivalent to the universal
    cover of $Q(S^1)$.
\end{theorem}

\begin{rem}
    In future work we show that $B\Csp$ is in fact contractible.
    Together with Theorem~\ref{thm:Cosp-fib} this implies that $B\rCsp \simeq Q(S^2)$,
    and further that $B \hCsp \simeq \tau_{\ge 3}Q(S^2)$ is the $2$-connected cover of $Q(S^2)$.
\end{rem}

This cospan category is related to the cobordism model for $K$-theory constructed in \cite{RS19}.
There the authors construct a ``cobordism category'' $\Cob(\mcC)$ 
for any unpointed Waldhausen category $\mcC$ such that $\gO B\Cob(\mcC)$ 
is the algebraic K-theory of $\mcC$.
In the case where $\mcC$ is the category $\Fin$ of finite sets, their cobordism category
is equivalent to the subcategory $\iCsp \subset \Csp$ of those cospans
for which the left leg is an injection.
By the same techniques as used for Theorem~\ref{thm:Cosp-fib} we obtain a new proof
of their result \cite[Corollary 3.2]{RS19} in the case $\mcC= \Fin$.
\begin{corollary}\label{cor:RS-Fin}
    There is a commutative diagram of homotopy fiber sequences 
    \[
        \begin{tikzcd}
            Q(S^1) \ar[r, "F"] \ar[d, equal] & B\iCsp \ar[d, "I"] \ar[r, "Q"] 
                                             & B\riCsp \ar[d, "{I'}"]\\
            Q(S^1) \ar[r] & B\Csp \ar[r, "R \circ P"] 
                          & B\rCsp
        \end{tikzcd}
    \]
    Moreover, $B\riCsp$ is contractible and therefore $F$ is an equivalence.
\end{corollary}

%

\subsection*{Structure of the paper}

We begin by proving the Main Theorem for quasicategories as \ref{thm:qCatAdd} 
in section \ref{sec:quasi}.
This section also contains the Key-Lemma \ref{lem:key}, 
which shows that the core idea of Steimle's proof remains true under our weaker assumptions.
Section~\ref{sec:qPasting} recalls the pasting lemma for homotopy pullbacks 
and proves several corollaries.
This is used in section~\ref{sec:semiSegal} to prove the main theorem for semi-Segal spaces
from its version for quasicategories. 
In doing so, we repair a proof of \cite{St17}.
Finally, section~\ref{sec:wutCat} uses the results of the previous sections
to prove the main Theorem~\ref{thm:wu-locadd}/\ref{thm:main} for weakly unital topological categories.

The final section~\ref{sec:Cosp} applies these techniques to the category of cospans.
We prove Theorem~\ref{thm:Cosp-fib}
by first constructing an explicit topological model for $\Csp$ 
and then applying Theorem~\ref{thm:wu-locadd}.
Then, the same setup is used to prove Corollary \ref{cor:RS-Fin}.

\subsection*{Acknowledgements}
I would like to thank my advisor Ulrike Tillmann 
for her support throughout all stages of the writing process. 
I would also like to thank Wolfgang Steimle for pointing me to \cite{RS19}
as well as for other helpful comments on an earlier draft of this paper.
I am very grateful for the support by
St. John's College, Oxford through the
``Ioan and Rosemary James Scholarship''.

\section{Realisation fibrations for quasicategories}\label{sec:quasi}

In proving the main theorem we will mostly follow Steimle's strategy, 
but generalise all necessary steps to the locally (co)Cartesian situation.
The first, and maybe most crucial step is the following generalisation 
of \cite[Theorem 2.14]{St18}, which might be of independent interest:

\begin{thm}[Main Theorem for quasicategories]\label{thm:qCatAdd}
    If a map $p:X \to Y$ of simplicial sets is a \emph{locally} Cartesian 
    and a \emph{locally} coCartesian fibration,
    then it is a realisation fibration.

    This means that for 
    any simplicial map $f:Z \to Y$ the following diagram is 
    a homotopy pullback in the Quillen model structure on simplicial sets:
    \[
        \begin{tikzcd}
            Z \times_Y X \ar[r] \ar[d] & X \ar[d, "p"]\\
            Z  \ar[r, "f"] & Y.
        \end{tikzcd}
    \]
\end{thm}

We recall the necessary definitions related to quasicategories from \cite{LurHTT} below.

\begin{defn}
    Recall that a map $p:X \to Y$ of simplicial sets is called a \emph{weak homotopy equivalence}
    if its geometric realisation $p: |X| \to |Y|$ is a weak equivalence of spaces.
    These are the weak equivalences in the Quillen model structure on simplicial sets.
\end{defn}

\begin{defn}
    A map $p:X \to Y$ of simplicial sets is called a \emph{realisation fibration}
    if for any simplicial map $f:Z \to Y$ 
    the following is a homotopy pullback square of spaces:
    \[
        \begin{tikzcd}
            {|Z \times_Y X|} \ar[r] \ar[d] & {|X|} \ar[d, "p"]\\
            {|Z|} \ar[r, "f"] & {|Y|}.
        \end{tikzcd}
    \]
\end{defn}

\begin{rem}
    A map of simplicial sets $p$ is a \emph{realisation fibration}
    if the map $Z \times_Y X \to Z \times_Y^h X$
    from the pullback to the homotopy pullback in the Quillen model structure 
    is a weak homotopy equivalence.
    In particular every Kan-fibration is a realisation fibration.
\end{rem}

We need some notation to denote fibers over simplices:
\begin{defn}
    For any map $p:X \to Y$ and simplex $\gs\in Y_n$ we write
    \[
        p_{|\gs}: (X_{|\gs} := X \times_{Y} \gD^n) \longrightarrow \gD^n
    \]
    to denote the pullback along $\gs:\gD^n \to Y$.
\end{defn}

Just like in \cite{St18} the following lemma is essential for our proof. 
It characterizes realisation fibrations in terms of their fibers over simplices.
\begin{lem}[{\cite[Lemma  1.4.B]{Wal85}, see also \cite{St18} and \cite{Rez14}}]%
    \label{lem:realChar}
    A map of simplicial sets $p:X \to Y$ is a realisation fibration 
    if and only if for any $n$-simplex $\gs: \gD^n \to Y$ 
    with first and last simplex $\gs(0), \gs(n): \gD^0 \to Y$ 
    the following inclusions are weak homotopy equivalences:
    \[
        X_{|\gs(0)} \longrightarrow X_{|\gs} \longleftarrow X_{|\gs(n)}.
    \]
\end{lem}
\begin{proof}
    Waldhausen deduces this from Quillen's \cite[Theorem~B]{Qui73}.
    A direct proof can be found in \cite[Lemma 3.4]{St18},
    or in the unpublished manuscript \cite{Rez14},
    where a more general version for simplicial spaces is proven.
    Note that the ``only if'' direction is trivial.
\end{proof}

In order to prove theorem \ref{thm:qCatAdd} we will need to recall 
some machinery from Lurie's ``Higher Topos Theory'', in particular \cite[section 2.4]{LurHTT}.
We will not attempt to give an introduction to the theory of quasicategories,
but only establish the necessary notation.
\begin{defn}
    A simplicial map $p:X \to Y$ is called an \emph{inner fibration}
    if it satisfies the horn-lifting condition for all inner horns $\gL_k^n$, $0 < k < n$.
    The map $p$ is a \emph{trivial Kan-fibration} 
    if it has the right lifting property with respect to all inclusions
    $\partial \gD^n \subset \gD^n$, $n \ge 0$.
    A \emph{quasicategory} is a simplicial set $X$ such that $X \to *$ is an inner fibration.
\end{defn}

\begin{defn}
    For a simplicial set $X$ and a simplex $\gs \in X_n$ the \emph{under-category}
    $X_{\gs/}$ of $X$ at $\gs$ has as set of $l$-simplices 
    \[
        (X_{\gs/})_l := \{ f: \gD^{n+l+1} \to X \;|\; f_{|[0,\dots,n]} = \gs \}.
    \]
    The projection map $X_{\gs/} \to X$ restricts $f$ to 
    $[n+1, \dots, n+l+1] \subset \gD^{n+l+1}$.
\end{defn}

\begin{rem}\label{rem:over-cat}
    Of course the above does not quite define the simplicial set $X_{\gs/}$:
    we still have to specify face and degeneracy maps. 
    We refer the reader to \cite[Proposition 1.2.9.2]{LurHTT} for more details.

    Consider the case when $X = N\mcC$ is the nerve of an ordinary category 
    and $\gs: [n] \to \mcC$ represents some simplex $\gs \in X_n$.
    Then the simplicial set $X_{\gs/}$ is canonically isomorphic to the nerve
    of the ordinary under category $\mcC_{\gs(n)/}$ with objects 
    $(c \in \mcC, g: \gs(n) \to c)$.
\end{rem}

\begin{defn}\label{defn:qCart}
    An edge $(f:x \to y):\gD^1 \to X$ is called \emph{$p$-coCartesian} 
    with respect to some $p:X \to Y$ if the following map is a trivial Kan fibration:
    \[
        X_{f/} \to X_{x/} \times_{Y_{p(x)/}} Y_{p(f)/}.
    \]
    An edge $f:\gD^1 \to X$ is called \emph{locally $p$-coCartesian} if it 
    is a coCartesian edge for $p_{|\gs}:X_{|\gs} \to \gD^1$ 
    where $\gs = p \circ f : \gD^1 \to Y$.

    We say that $p:X \to Y$ is a \emph{(locally) coCartesian fibration}, 
    if it is an inner fibration and for every edge $\cl{f}:\cl{x} \to \cl{y}$ in $Y$ 
    and $x\in X$ with $p(x) = \cl{x}$ there is a (locally) $p$-coCartesian $f:x \to y$ in $Y$ 
    such that $p(y) = \cl{y}$ and $p(f) = \cl{f}$.
\end{defn}

\begin{rem}\label{rem:locally-Cartesian}
    A simplicial map $p:X \to Y$ is a locally coCartesian fibration
    if and only if it is an inner fibration and for every $\gs: \gD^1 \to Y$ the pullback
    $ p_{|\gs}:X_{|\gs} \to \gD^1 $
    is a coCartesian fibration.
\end{rem}

\subsection{Proof of Theorem \ref{thm:qCatAdd}}

We will use Waldhausens's criterion (lemma \ref{lem:realChar}) to prove theorem \ref{thm:qCatAdd}. 
Unlike \cite{St18}, we will do not check the criteria directly, 
but rather use Quillen's theorem A for quasicategories.
Our Key-Lemma \ref{lem:key} shows that when applying Quillen A to the fibers over a simplex,
we in fact only use that the map in question is a \emph{locally} (co)Cartesian fibration.
This is the most essential step of our generalisation.

\begin{thm}[Quillen's theorem A for quasicategories, {\cite{LurHTT}}]
    \label{thm:qCatQuillenA}
    If $p:X \to Y$ satisfies that $X \times_{Y} Y_{y/}$ is weakly contractible for any $y \in Y$
    and $Y$ is a quasicategory,
    then $p$ is a weak homotopy equivalence.
\end{thm}
\begin{proof}[Remark on the reference]
    \cite[Theorem 4.1.3.1]{LurHTT} states that under the above assumptions
    $p$ has to be a cofinal map and 
    \cite[Proposition 4.1.1.3.(3)]{LurHTT}
    states that cofinal maps are weak equivalences in the Quillen model structure.
\end{proof}

We now proceed to check that for coCartesian fibrations over $\gD^n$ 
the inclusion of the fiber of the last vertex $n \in \gD^n$ 
into the total space satisfies the conditions of Quillen's Theorem A.
\begin{lem}\label{lem:CartQc}
    For any coCartesian fibration $p:X \to \gD^n$ 
    and $x \in X$ the simplicial set $X_{|n} \times_{X} X_{x/}$ is weakly contractible.
\end{lem}
\begin{proof}
    Using remark \ref{rem:over-cat} note that $(\gD^n)_{f/} \cong (\gD^n)_{n/} \cong \{n\}$ holds 
    for the unique morphism $f:p(x) \to n$ in $\gD^n$.
    We use this to write 
    \[
        X_{|n} = X \times_{\gD^n} \{n\} \cong X \times_{\gD^n} (\gD^n)_{f/}.
    \]
    Therefore
    \[
        X_{x/} \times_X X_{|n}  \cong X_{x/} \times_X (X \times_{\gD^n} (\gD^n)_{f/})
        \cong X_{x/} \times_{\gD^n} (\gD^n)_{f/}.
    \]
    Both of the maps $X_{x/} \to \gD^n$ and $(\gD^n)_{f/}\to \gD^n$ factor through the inclusion
    $(\gD^n)_{k/}\cong \{k, \dots, n\} \subset \gD^n$ for $k=p(x)$, 
    so we can further rewrite this as
    \[
        X_{x/} \times_X X_{|n} \cong X_{x/} \times_{\gD^n} (\gD^n)_{f/}
        \cong X_{x/} \times_{(\gD^n)_{k/}} (\gD^n)_{f/}
    \]
    Let now $\cl{f}:x \to x'$ be a $p$-coCartesian lift of $f:k \to n$.
    Then, by definition of coCartesian lifts, the following map is a trivial Kan-fibration
    \[
        X_{\cl{f}/} \longrightarrow X_{x/} \times_{\gD^n_{p(\cl{x})/}} (\gD^n)_{p(\cl{f})/}
            \cong X_{x/} \times_{X} X_{|n}.
    \]
    The left-hand-side is weakly contractible as it has the initial object $\cl{f}$.
    But trivial Kan-fibrations are in particular weak homotopy equivalence,
    hence also the right-hand-side is weakly contractible.
\end{proof}

In fact, the previous lemma holds in the more general case 
of \emph{locally} coCartesian fibrations, as we shall now show.
This is the key observation in generalising Steimle's theorem.
    
\begin{lem}[Key-Lemma]\label{lem:key}
    For any \emph{locally} coCartesian fibration $p:X \to \gD^n$ 
    and $x \in X$ the simplicial set $X_{|n} \times_{X} X_{x/}$ is weakly contractible.
\end{lem}
\begin{proof}
    We will prove this using lemma \ref{lem:CartQc}.
    Write $k:=p(x)$ and $\gs:\gD^1 \to \gD^n$ for the edge 
    given by $\gs(0) = k$ and $\gs(1) = n$.
    Since 
        $p_{|\gs}: X_{|\gs} \to \gD^1$
    is a coCartesian fibration,
    the previous lemma tells us that
    \[ 
        (X_{|\gs})_{|1} \times_{X_{|\gs}} (X_{|\gs})_{(x,0)/}
    \]
    is weakly contractible, where $(x,0)\in X \times_{\gD^n} \gD^1 = X_{|\gs}$.
    We will argue that the map
    \[
        (X_{|\gs})_{|1} \times_{X_{|\gs}} (X_{|\gs})_{(x,0)/} \to X_{|n} \times_{X} X_{x/}
    \]
    is an isomorphism.
    Clearly $(X_{|\gs})_{|1} \cong X_{|n}$.

    More importantly, consider the simplicial map
    $
        s: (X_{|\gs})_{(x,0)/} 
        \longrightarrow 
        (X_{x/})_{|\gs}
    $.
    We check ``by hand'' that it induces a bijection on $l$-simplices
    and hence is an isomorphism of simplicial sets:
    \begin{align*}
        ((X_{|\gs})_{(x,0)/})_{l} &= \{a:\gD^{l+1} \to X_{|\gs} \;|\; a(0) = (x,0)\}\\
                                 &\cong \{b:\gD^{l+1} \to X, c:\gD^{l+1} \to \gD^1 \;|\; 
                                    b(0) = x, c(0) = 0, p\circ b = \gs \circ c\}\\
                                 &\cong \{b:\gD^l \to X_{x/}, c:\gD^l \to \gD^1 \;|\;
                                    p \circ b = \gs \circ c\} 
                                 = ((X_{x/})_{|\gs})_l.
    \end{align*}
    Using this we see that
    \begin{align*}
        (X_{|\gs})_{|1} \times_{X_{|\gs}} (X_{|\gs})_{(x,0)/} 
        &\cong X_{|n} \times_{X_{|\gs}} (X_{x/})_{|\gs}
        \cong X_{|n} \times_{X_{|\gs}} (X_{|\gs} \times_X X_{x/})\\
        &\cong X_{|n} \times_X X_{x/}
    \end{align*}
    which is the desired isomorphism.
    As mentioned in the beginning of the proof the left-hand-side is contractible 
    by lemma \ref{lem:CartQc} and therefore so is the right-hand-side,
    proving the claim.
\end{proof}

From the above main theorem for quasicategories follows immediately.
\begin{proof}[Proof of theorem \ref{thm:qCatAdd}]
    We need to verify the conditions of lemma \ref{lem:realChar}.
    In fact it suffices to check the second equivalence, 
    the first one then follows by applying the same argument to $p^{op}$.
    But this is exactly the combination of the key lemma \ref{lem:key}
    and Quillen's theorem A \ref{thm:qCatQuillenA}.
\end{proof}

\section{A cubical pasting lemma for homotopy Cartesian squares}\label{sec:qPasting}

While the definition of a Cartesian edge \ref{defn:Cart} asks for a certain diagram 
to be homotopy Cartesian, its local counter-part only asks for two specific homotopy fibers
to be equivalent.
We now formalise this notion:
\begin{defn}\label{defn:locCartSqr}
    We say that a commutative diagram of spaces 
    \[
        \begin{tikzcd}
            A \ar[r, "i"] \ar[d, "p"] & B \ar[d, "q"] \\
            C \ar[r, "j"] & D 
        \end{tikzcd}
    \]
    is \emph{homotopy Cartesian at a point} $c \in C$, if the induced map on homotopy fibers
    \[
        \hofib_c(p) \xrightarrow{\ i\ } \hofib_{j(c)}(q)
    \]
    is a weak equivalence.
\end{defn}

Recall that the homotopy fiber $\hofib_y(p)$ of a map $p:X \to Y$ at a point $y \in Y$
is defined as the space of tuples $(x, \gc)$ where $x \in X$ and $\gc$ is a path from $p(x)$ to $y$.
Note that if a square is homotopy Cartesian at $c \in C$ and $c' \in C$ is path-connected to $c$,
then the square is also homotopy Cartesian at $c'$,
since the homotopy type of the homotopy fiber does not change as we move the point.
Moreover, this indeed generalises the notion of homotopy Cartesian squares:
\begin{lem}[{\cite[Proposition 3.3.18]{MV15}}]\label{lem:point-wise-Cartesian}
    A square $(ABCD)$ as above is homotopy Cartesian if and only if it is homotopy Cartesian
    at all points $c \in C$.
\end{lem}


A key consequence of this is the pullback pasting lemma, which allows us to glue and sometimes cancel 
homotopy Cartesian squares:
\begin{lem}[Pullback pasting lemma]\label{lem:pb-pasting}
    For a commutative diagram of spaces
    \[
     \begin{tikzcd} 
        A \ar[r] \ar[d] \ar[rd, phantom, "*"] & B \ar[r] \ar[d] \ar[rd, phantom, "\dagger"] & C \ar[d]\\
        D \ar[r, "f"] & E \ar[r, "g"] & F 
     \end{tikzcd}
    \]
    the following hold:
    \begin{itemize}
     \item If $\dagger$ is homotopy Cartesian, then $*$ is homotopy Cartesian 
     if and only if the rectangle $*\dag$ is homotopy Cartesian.
     \item If $*$ and $*\dag$ are homotopy Cartesian and $f$ is surjective on path-components,
        then $\dag$ is homotopy Cartesian.
    \end{itemize}
\end{lem}
\begin{proof}
    For every $d \in D$ there is a commutative triangle of homotopy fibers:
    \[
    \begin{tikzcd}[column sep = small]
        \hofib_d(A \to D) \ar[rr] \ar[dr] && \hofib_{(g \circ f)(d)}(C \to F) \\
        & \hofib_{f(d)}(B \to E) \ar[ur] & 
    \end{tikzcd}
    \]
    and hence if two of these maps are equivalences, then so is the third.
    We will use this to prove the second point. The first one follows by a similar, but easier, argument.
    
    Assume that $*$ and $*\dag$ are homotopy Cartesian and that $f$ is surjective on path-components.
    To show that $\dag$ is homotopy Cartesian it suffices by lemma \ref{lem:point-wise-Cartesian}
    to show that it is homotopy Cartesian at every point $e \in E$. 
    Since $f$ is surjective on path-components, we can find $d \in D$ such that there is a path 
    from $f(d)$ to $e$. Using the $2$-out-of-$3$ property on the commutative triangle
    at the start of the proof we see that because $*$ is homotopy Cartesian at $d$
    and $*\dag$ is homotopy Cartesian at $d$, $\dag$ has to be homotopy Cartesian at $f(d)$.
    Since there is a path from $f(d)$ to $e$, $\dag$ is homotopy Cartesian at $e$, as claimed.
\end{proof}


We will need the following cubical consequence of pasting lemma:
\begin{cor}\label{cor:cubical-po-law}
   Consider a diagram of cubical shape:
    \[
    \begin{tikzcd}[column sep = small, row sep = small]
    & C \arrow[rr] \arrow[dd] &&  D \arrow[dd] \\
    A \arrow[dd] \arrow[rr, crossing over] \arrow[ru] && B \arrow[ru, "f"'] &\\
    & Y \arrow[rr] & & Z\\
    W \arrow[ru] \arrow[rr] & & X \arrow[ru] \arrow[from=uu, crossing over] &             
    \end{tikzcd}
    \]
    where the top and the bottom square are homotopy Cartesian.
    \begin{itemize}
        \item If the back square is homotopy Cartesian, then so is the front square.
        \item If the front square is homotopy Cartesian and the map $f:B \to D$ is surjective on path-components,
        then the back square is homotopy Cartesian.
    \end{itemize}
\end{cor}
\begin{proof}
    For the first point assume that the squares $(ABCD)$, $(WXYZ)$ and $(CDYZ)$ are homotopy Cartesian.
    Then the pasting lemma implies that the rectangle $(ABYZ)$, which is glued from $(ABCD)$ and $(CDYZ)$,
    is also homotopy Cartesian. Moreover, this rectangle can also be obtained by gluing 
    $(ABWX)$ and $(WXYZ)$. Applying the pasting lemma to this shows that $(ABWX)$ is homotopy Cartesian.
    
    The second part is proven by reversing the argument, with the notable difference being
    that we need the path-component surjectivity of $f$ to see that
    $(ABCD)$ and $(ABYZ)$ being homotopy Cartesian implies that $(CDYZ)$ is homotopy Cartesian.
\end{proof}

Finally, the following lemma allows us to check whether a square is homotopy Cartesian at a point,
by passing to homotopy fibers with respect to some reference map $(E \to B)$.
This will be needed to compare the definition of locally Cartesian edges in weakly unital topological 
categories with the one for semi-Segal spaces.
\begin{lem}\label{lem:local-Cartesian-via-fibers}
    Consider a commutative diagram of spaces where $p$ is a Serre fibration:
    \[
        \begin{tikzcd}
            W \ar[r] \ar[d] \ar[dr, phantom, "*"] & X \ar[r] \ar[d] & E \ar[d, "p"] \\
            Y \ar[r, "f"] & Z \ar[r, "g"] & B
        \end{tikzcd}
    \]
    Fix $y \in Y$, set $b := (g \circ f)(y) \in B$, and write $\wt{y} \in \hofib_b(Y \to B)$
    for $y$ equipped with the constant path.
    Then the left square above $(*)$ is homotopy Cartesian at $y$, 
    if and only if for all lifts $e \in p^{-1}(b)$
    the following square of homotopy fibers is homotopy Cartesian at $\wt{y}$:
    \[
        \begin{tikzcd}
            \hofib_e(W \to E) \ar[r] \ar[d] \ar[dr, phantom, "\dag_e"] & 
            \hofib_e(X \to E) \ar[d] \\
            \hofib_b(Y \to B) \ar[r] & \hofib_b(Z \to B)
        \end{tikzcd}
    \]
\end{lem}
\begin{proof}
    We may assume without loss of generality that the space $Y$ is path-connected:
    If it is not, we simply replace $Y$ by the path-component of $y$ 
    and $W$ by the preimage of this path-component.
    This does not affect any of the hypotheses of the lemma, though it does lead to the simplification 
    that $(*)$ is homotopy Cartesian at $y$ if and only if it is homotopy Cartesian.
    By a similar argument we may also assume that $B$ is path-connected.
    
    We begin by applying corollary \ref{cor:cubical-po-law} to the following cube:
    \[
    \begin{tikzcd}[column sep = small, row sep = small]
    & W \arrow[rr] \arrow[dd] &&  X \arrow[dd] \\
    \coprod_{e \in p^{-1}(b)} \hofib_e(W \to E) \arrow[dd] \arrow[rr, crossing over] \arrow[ru] && 
            \coprod_{e \in p^{-1}(b)} \hofib_e(X \to E)
    \arrow[ru, "f"'] &\\
    & Y \arrow[rr] & & Z\\
    \hofib_b(Y \to B) \arrow[ru] \arrow[rr] & & \hofib_b(Z \to B) \arrow[ru] \arrow[from=uu, crossing over] &             
    \end{tikzcd}
    \]
    The top square is seen to be homotopy Cartesian by applying the pasting lemma
    to the following diagram:
    \[
        \begin{tikzcd}
            \coprod_{e \in p^{-1}(b)} \hofib_e(W \to E) \ar[r] \ar[d] & 
            \coprod_{e \in p^{-1}(b)} \hofib_e(X \to E) \ar[r] \ar[d] & 
            \coprod_{e \in p^{-1}(b)} \{e\} \ar[d] \\
            W \ar[r] & X \ar[r] & E
        \end{tikzcd}
    \]
    and a similar diagram shows that the bottom square is homotopy Cartesian. 
    Note that the map $\coprod_{e \in p^{-1}(b)} \hofib_e(X \to E) \to X$ is surjective 
    on path-components because we assumed that $B$ is path-connected.
    Corollary \ref{cor:cubical-po-law} hence implies that the back of the cube $(*)$
    is homotopy Cartesian if and only if the front of the cube $(\coprod_e \dag_e)$ 
    is homotopy Cartesian.
    This directly implies the ``if" direction of the lemma.
    
    For the ``only if" direction, assume that $(\dag_e)$ is homotopy Cartesian at $\wt{y}$ 
    for all $e \in p^{-1}(b)$.
    It remains to show that $(\coprod_e \dag_e)$ is homotopy Cartesian, 
    or equivalently that $(\dag_e)$ is homotopy Cartesian at all points $y' \in \hofib_b(Y \to B)$.
    For any such $y'$ we can find a loop $\gc \in \Omega_b B$ in $B$
    that is based at $b$, such that $y'$ is in the same path component as $\gc . \wt{y}$,
    i.e.\ as $(y, \gc) \in Y \times_B B^{[0,1]} \times_B \{b\} = \hofib_{b}(Y,B)$.
    This is possible because $Y$ is assumed to be connected and hence $\pi_1(B,b)$
    acts transitively on $\pi_0 \hofib_b(Y \to B)$.
    Let $\wt{\gc}$ be a path in $E$ that starts at $e$ and lifts $\gc$.
    Concatenating with $\wt{\gc}$ and $\gc$, respectively, yields an equivalence of squares:
    \[
    \begin{tikzcd}[column sep = small, row sep = small]
    & \hofib_{e'}(W \to E) \arrow[rr] \arrow[dd] &&  \hofib_{e'}(X \to E) \arrow[dd] \\
    \hofib_e(W \to E) \arrow[dd] \arrow[rr, crossing over] \arrow[ru, "\wt{\gc}_*", "\simeq"'] && 
            \hofib_e(X \to E)
    \arrow[ru, "\wt{\gc}_*", "\simeq"'] &\\
    & \hofib_b(Y \to B) \arrow[rr] & & \hofib_b(Z \to B)\\
    \hofib_b(Y \to B) \arrow[ru, "\gc_*", "\simeq"'] \arrow[rr] & & \hofib_b(Z \to B) \arrow[ru, "\gc_*", "\simeq"'] \arrow[from=uu, crossing over] &             
    \end{tikzcd}
    \]
    Here $e' = \wt{\gc}_*e$ is the end-point of the lifted path $\wt{\gc}$.
    The front square in this diagram is $(\dag_e)$ and the back square is $(\dag_{e'})$.
    Since the diagrams are equivalent, we can conclude that $(\dag_{e'})$ is homotopy Cartesian at 
    $y' = \wt{\gc}_*\wt{y}$ because $(\dag_e)$ is assumed to be homotopy Cartesian at $\wt{y}$.
    This shows that $(\dag_{e'})$ is homotopy Cartesian at $y'$ for all $y' \in \hofib_e(Y \to E)$
    and $e' \in p^{-1}(b)$. 
\end{proof}

\section{Realisation fibrations and semi-Segal spaces}\label{sec:semiSegal}
Before can we prove the additivity theorem for weakly unital categories,
we first need to establish an intermediate version for semi-Segal spaces.
Most importantly, we will introduce a notion of locally Cartesian 
fibrations between semi-Segal spaces and check that it is compatible 
with the notion we have for quasicategories. 

\subsection{Weakly unital semi-Segal spaces}
We now introduce the notion of a semi-Segal space and define weak units following \cite{Har15}.
Unlike Harpaz we will not require Reedy fibrancy throughout,
and instead give definitions that are invariant under level-wise weak equivalences.
(See lemma \ref{lem:weak-unit-invariant} for the compatibility.)
Note that this approach is different to the one taken in \cite{St18} where 
weakly unital semi-Segal spaces are defined in terms of equivalences instead of weak units.
This is not an option for us as the definition of a locally Cartesian edge 
relies on a notion of weak unit.
However, it follows from work of Harpaz that the approaches are equivalent, see fact \ref{fact:Harpaz}.

\begin{defn}
    The semi-simplicial category $\gD^{inj} \subset \gD$ has all objects,
    but only those maps that are injective.
    We write $s\gD^n$ for the semi-simplicial set represented by $[n] \in \gD^{inj}$.
    It is a sub-semi-simplicial set of $\gD^n$ and contains precisely those 
    simplices that are not degenerate.
\end{defn}

\begin{defn}
    A semi-simplicial space $X$ is a \emph{semi-Segal space},
    if the natural map
    \[
        X_n \to X_1 \times_{X_0}^h \dots \times_{X_0}^h X_1
    \]
    is a weak equivalence for all $n \ge 2$.
\end{defn}

\begin{defn}\label{defn:X(A)}
    For a semi-simplicial space $X:(\gD^{inj})^{op} \to \mi{Top}$
    and a semi-simplicial set $A$ we define
    \[
        X(A) = \colim_{s\gD^n \to A} X_n
    \]
    where the colimit is taken over all $n\ge 0$ and all $n$-simplices of $A$, ordered by inclusion.
    This definition is such that $X(s\gD^n) \cong A_n$.
    
    Moreover, if $x \in X_k$ we define $X_n/x$ to be the homotopy fiber, at the point $x$,
    of the map $X_n \to X_k$ induced by $[k]\cong \{n-k, \dots, n\} \subset \{0, \dots, n\}= [n]$.
    Similarly, we let $x\backslash X_n$ denote the homotopy fiber, at the point $x$,
    of the map $X_n \to X_k$ induced by $[k] = \{0, \dots, k\} \subset \{0, \dots, n\}= [n]$.
\end{defn}

\begin{defn}
    A map $f: X \to Y$ of semi-simplicial spaces is a \emph{Reedy-fibration},
    if for any inclusion $A \subset B$ of semi-simplicial sets the 
    induced map
    \[
        X(A) \to X(B) \times_{Y(B)} Y(A).
    \]
    is a Serre fibration.
\end{defn}

    To define weak units, consider for a semi-Segal space $X$ and a vertex $x\in X_0$ 
    the homotopy pullback $X_1 \times_{X_0^2}^h \{(x,x)\} = x \backslash X_1 /x$.
    A point in this space is given by a triple $(e, \gc_0, \gc_1)$ of an edge $e \in X_1$ and
    paths $\gc_i$ in $X_0$ from $x$ to $d_ie$.
    This can be thought of as the space of endomorphisms of $x$ in $X$.
\begin{defn}\label{defn:weak-unit}
    A triple $(u,\gc_0,\gc_1) \in x \backslash X_1 /x$ is called a \emph{weak unit} 
    if for all $a \in X_0$ the following diagrams commute up to homotopy (after cofibrantly replacing each of the spaces):
    \[
    \begin{tikzcd}
        a\backslash X_2/u \ar[r, "d_2"] \ar[d, "d_1"'] & a\backslash X_1/d_1u \ar[d, "*\gc_1"] \\
        a\backslash X_1/d_0u \ar[r, "*\gc_0"] & a\backslash X_1/x
    \end{tikzcd}
    \quad
    \begin{tikzcd}
        u\backslash X_2 /a \ar[r, "d_0"] \ar[d, "d_1"'] & d_0u \backslash X_1 /a \ar[d, "\gc_0*"] \\
        d_1u\backslash X_1 /a \ar[r, "\gc_1*"] &  x\backslash X_1 /a
    \end{tikzcd}
    \]
    Here $*\gc$ and $\gc*$ denote concatenation with a path $\gc$ from the right or left, respectively.
\end{defn}

\begin{defn}
    For a semi-Segal space we let $X^{\rm wu}_x \subset x \backslash X_1 / x$
    denote the subspace of weak units at $x$.
    We say that $X$ is \emph{weakly unital} if $X^{\rm wu}_x$ is non-empty for all $x \in X_0$.
    We say that a semi-simplicial map $f:X \to Y$ between weakly unital semi-Segal spaces 
    is \emph{weakly unital} if it satisfies $f(X^{\rm wu}_x) \subset Y^{\rm wu}_{f(x)}$ for all $x \in X_0$.
%
\end{defn}

In the two diagrams all maps except for $d_1$ are equivalences for all edges $u$,
and the map $d_1$ can be thought of as post- and pre-composition with $u$, respectively. 
Hence the condition that the diagrams commute up to homotopy encodes the idea 
that both post- and pre-composition with $u$ are homotopic to the identity.
Moreover, it follows that $u$ is an equivalence in the sense of \cite[Definition 2.9]{St18}.

\begin{lem}\label{lem:weak-unit-invariant}
    The space of weak units $X^{\rm wu}_x \subset x\backslash X_1 /x$ has the following invariance properties:
    \begin{itemize}
        \item[(i)] For all $x \in X_0$ the subspace $X^{\rm wu}_x \subset x\backslash X_1 /x$ is a union of path-components.
        \item[(ii)] If $x, x' \in X_0$ are in the same path-component, then $X^{\rm wu}_x \simeq X^{\rm wu}_{x'}$.
        \item[(iii)] A level-wise equivalence $f:X \to Y$ induces equivalences $X^{\rm wu}_x \simeq Y^{\rm wu}_{f(x)}$.
        \item[(iv)] If $X$ is a Reedy fibrant semi-Segal space and $Z := \mathrm{Sing}_*(X)$ the semi-simplicial Kan complex obtained by taking level-wise singular complexes, 
        then $Z$ is a semi-Segal space in the sense of \cite{Har15} and for each $x \in X_0$ there is a canonical equivalence 
        $X^{\rm wu}_x \simeq |Z_x^{\rm qu}|$ 
        between the space of weak units at $x$ defined above and Harpaz's space of quasi-units at $x$.
    \end{itemize}
\end{lem}
\begin{proof}
    The property of two maps being homotopic is invariant under weak equivalence in the following sense:
    Let $X, Y, X', Y'$ be cofibrant spaces and let $f, g:X \to Y$ and $f',g':X' \to Y'$ be maps such that there are weak equivalences $w:X \to X'$ and $v:Y \to Y'$ for which $f' \circ w \simeq v \circ f$ and $g' \circ w \simeq v \circ g$.
    Then $f \simeq g$ if and only if $f' \simeq g'$.
    To prove this, pass to the homotopy category.
    Since all spaces involved are cofibrant (and all spaces are fibrant) $v$ and $w$ are isomorphisms in the homotopy category.
    By writing $[f'] = [v] \circ [f] \circ [w]^{-1}$ it follows easily that $[f] = [g]$ if and only if $[f'] = [g']$.
    
    Using this fact, points (i-iii) follow, because each of them amounts to a replacing the square in definition \ref{defn:weak-unit} by a weakly equivalent square.
    
    For (iv) consider a Reedy fibrant semi-Segal space $X$ and let $Z := \mathrm{Sing}_*(X)$ be as above.
    Then $Z$ is Reedy fibrant because the singular complex functor $\mathrm{Sing}_*:\mathrm{Top} \to \mathrm{sSet}$ preserves Kan-fibrations.
    Moreover, $Z$ satisfies the Segal condition because $\mathrm{Sing}_*$ preserves homotopy pullback squares.
    To study the space of weak units in $X$, first note that we may replace the homotopy pullback $x\backslash X_1 /x$ by the genuine fiber of $X_1 \to X_0 \times X_0$ at $(x,x)$ since this map is a Serre fibration for Reedy fibrant $X$.
    It remains to show that an edge $e \in X_1$ with $d_0 e= d_1 e = x$ is a weak unit if and only if the corresponding edge $e \in Z_1$ is a quasi-unit.
    The definition of a weak unit says that the following two pairs of maps are homotopic for all $a$:
    \[\begin{tikzcd}
        a \backslash X_2 / e \ar[r, bend left, "d_2"] \ar[r, bend right, "d_1"'] &
        a \backslash X_1 /x
    \end{tikzcd}
    \quad
    \begin{tikzcd}
        e \backslash X_2 / a \ar[r, bend left, "d_0"] \ar[r, bend right, "d_1"'] &
        x \backslash X_1 /a
    \end{tikzcd}\]
    The definition of a quasi-unit \cite[Definition 1.4.11]{Har15} says that $[e]_*:\Map_X(a,x) \to \Map_X(a,x)$ and $[e]^*:\Map_X(a,x) \to \Map_X(a,x)$ are the identity morphism in the homotopy category.
    Inspecting the definition we see that $\Map_X(a, x) \cong a \backslash X_1 /x$ and that $[e]_* := [d_2] \circ [d_1]^{-1}$.
    Hence $[e]_* = \id_{\Map_X(a,x)}$ holds if and only if $[d_2] = [d_1]$, i.e.\ if $d_2, d_1: a \backslash X_2 / e \to a \backslash X_1/x$ are homotopic.
    The analogous statement holds for $[e]^*$ and it follows that $e$ is a weak unit if and only if it is a quasi-unit.
\end{proof}

Since every semi simplicial space has a weakly equivalent Reedy-fibrant replacement we can use a combination of point (iii) and (iv) in lemma \ref{lem:weak-unit-invariant} to translate Harpaz's results to our setting.
\begin{fact}[Lemma 1.4.16, Lemma 1.5.5, Proposition 1.5.6 of {\cite{Har15}}]\label{fact:Harpaz} 
\ 
\begin{itemize}
    \item[(i)] 
    For every semi-Segal space $X$ and $x \in X_0$ the space $X^{\rm wu}(x)$ is either connected or empty.
    \item[(ii)]
    A semi-Segal space $X$ is weakly unital,
    if and only if $X$ is weakly unital in the sense of \cite[Definition 2.9]{St17}
    \item[(iii)]
    A map of weakly unital semi-Segal spaces $f:X \to Y$ is weakly unital,
    if and only if $f$ is weakly unital in the sense of \cite[Definition 2.9]{St17}.
\end{itemize}
\end{fact}

\subsection{(Locally) co/Cartesian fibrations of semi-Segal spaces}

\begin{defn}\label{defn:sS-Cart}
    An edge $e \in X$ is $p$-Cartesian for $p:X \to Y$
    a map of semi-Segal spaces, if 
    the following square is homotopy Cartesian:
    \[
        \begin{tikzcd}
            X_2/e \ar[r, "d_1"] \ar[d, "p"] & X_1/d_0e \ar[d, "p"] \\
            Y_2/p(e) \ar[r, "d_1"] & Y_1/d_0p(e) 
        \end{tikzcd}
    \]
    We say that $e$ is locally $p$-Cartesian if the square is homotopy Cartesian
    at all $s \in Y_2/p(e)$ for which $d_2s \in Y_1/d_1p(e)$ can be lifted to a weak unit in $Y_{d_1p(e)}^{\rm wu}$.

    Write $X_1^{p \rm -(l)Cart} \subset Y_1$ for the subset of those edges 
    that are (locally) $p$-Cartesian.
    A map $p:X \to Y$ of semi-Segal spaces is a (locally) Cartesian fibration,
    if $(d_1, p):X_1^{p \rm -(l)Cart} \to X_0 \times_{Y_0}^h Y_1$ 
    is surjective on path-components.
    The coCartesian notions are defined dually.
\end{defn}

\begin{rem}\label{rem:weak-units-are-Cartesian}
    Note that while the definition of a locally $p$-Cartesian edges requires us to check that the square 
    is homotopy Cartesian at \emph{all} $s \in Y_2/p(e)$ for which $d_2s \in Y_1/d_1p(e)$ can be lifted to a weak unit in $Y_{d_1p(e)}^{\rm wu}$,
    it in fact suffices to do so at any single $s$.
    This will follow from lemma \ref{lem:point-wise-Cartesian} once we show that any other $s' \in Y_2/p(e)$
    satisfying this property is path-connected to $s$.
    Since $d_2:Y_2/p(e) \to Y_1/d_1p(e)$ is an equivalence by the Segal condition,
    it suffices to show that $d_2s$ and $d_2s'$ are in the same path-component of $Y_1/d_1p(e)$.
    By assumption both $d_2s$ and $d_2s'$ are in the image of the projection 
    $Y_{d_1p(e)}^{\rm wu} \subset d_1p(e) \backslash Y_1 / d_1p(e) \to Y_1 / d_1p(e)$.
    The space $Y_{d_1p(e)}^{\rm wu}$ is path-connected by fact \ref{fact:Harpaz}(i), so the claim follows.
\end{rem}

\begin{ex}\label{ex:weak-units-are-Cartesian}
    Consider a weakly unital map of weakly unital semi-Segal spaces $p:X \to Y$.
    Then for every weak unit $(e, \gc_0, \gc_1) \in x\backslash X_1/x$
    the underlying edge $e \in X_1$ is $p$-Cartesian.
    Indeed, the definition of a weak unit implies that $d_1:X_2/e \to X_1/d_0e$ is an equivalence,
    and since $P$ is assumed to preserve weak units  
    $d_1:Y_2/p(e) \to Y_1/d_0p(e)$ also is an equivalence.
    Therefore both horizontal maps in the square of definition \ref{defn:Cart} are equivalences,
    making it homotopy Cartesian.
    This shows that $e$ is Cartesian and hence the subspace $X_1^{p-\rm Cart} \subset X_1$
    contains the weak units.
    Because every object in $X$ is assumed to admit a weak unit and weak units are unique up to homotopy, 
    this implies that the map $(d_1,p):X_1^{p-\rm Cart} \to X_0 \times_{Y_0}^h Y_1$
    hits all components where the second coordinate is a weak unit.
\end{ex}

We will later need to use that, just like in the case of quasicategories, the restriction
of a locally Cartesian fibration to any edge is a Cartesian fibration.
\begin{lem}\label{lem:pullback-of-Cartesian}
    Let $P: X \to Y$ be a weakly unital functor between weakly unital semi-Segal spaces. 
    Assume further that $P$ is locally Cartesian and a Reedy fibration.
    Then for every weakly unital functor $\gs:\gD^1 \to Y$
    the pullback $P_{|\gs}:X \times_Y \gD^1 \to \gD^1$ is a Cartesian fibration.
\end{lem}
\begin{proof}
%
    There are three edges in $\gD^1$. The two degenerate ones are weak units and 
    hence by example \ref{ex:weak-units-are-Cartesian} they always admit Cartesian lifts.
    Now consider the non-degenerate edge $f \in (\gD^1)_1$ corresponding to $\id_{[1]}: [1] \to [1]$.
    Given a lift of the source $(x, i) \in (X \times_Y \gD^1)_0$ we need to find 
    a Cartesian lift of $f$. For $(x, i)$ to be a lift of $d_1 f = 0$ we need that $i=0$ and $P(x) = \gs(0)$. 
    Since $P$ is assumed to be locally Cartesian, we can find a locally $P$-cartesian edge $e \in X_1$
    such that $d_1 e = x$ and $P(e) = \gs(f)$.
    Hence $(e, f)$ is the desired lift of $f$, if we can show that it is $P_{|\gs}$-Cartesian.
    
    Consider the square from definition \ref{defn:Cart} applied to the edge $(e,f)$:
    \[
        \begin{tikzcd}
            (X \times_Y \gD^1)_2/(e,f) \ar[r, "d_1"] \ar[d, "p"] & 
            (X \times_Y \gD^1)_1/(d_0e, 1) \ar[d, "p"] \\
            (\gD^1)_2/f \ar[r, "d_1"] & (\gD^1)_1/d_0f
        \end{tikzcd}
    \]
    The space $(\gD^1)_2/f$ has a single point, namely the $2$-simplex $\rho:[2] \to [1]$ 
    that sends $0,1 \mapsto 0$ and $2 \mapsto 1$.
    The space $(\gD^1)_1/d_0f$ has two points: the edge $f$ and the degenerate edge at $d_0f =1$.
    Therefore, this square is a pullback if and only if $(X \times_Y \gD^1)_2/(e,f)$
    is equivalent to the part of $(X \times_Y \gD^1)_1/(d_0e, d_0f)$ that lies over 
    $f \in (\gD^1)_1/d_0f$.
    We may rewrite $(X \times_Y \gD^1)_2/(e,f) \cong X_2/e \times_{Y_2/\gs(f)} (\gD^1)_2/f$
    and $(X \times_Y \gD^1)_1/(d_0e, d_0f) \cong X_1/d_0e \times_{Y_1/d_0\gs(f)} (\gD^1)_1/d_0f$.
    Hence it suffices to show that in the following square the homotopy fibers of the vertical maps
    at $\gs(\rho) \in Y_1/d_0\gs(f)$ and $\gs(f) \in Y_1/d_0\gs(f)$ agree:
    \[
        \begin{tikzcd}
            X_2/e \ar[r, "d_1"] \ar[d, "p"] & 
            X_1/d_0e \ar[d, "p"] \\
            Y_2/\gs(f) \ar[r, "d_1"] & 
            Y_1/d_0\gs(f)
        \end{tikzcd}
    \]
    Since $d_2\gs(\rho) = \gs(d_2 \rho)$ is a weak unit this follows from the assumption 
    that $e$ is a locally $P$-cartesian edge.
\end{proof}

\subsection{The additivity theorem for semi-Segal spaces}
To transfer the result about quasicategorical realisation fibrations 
to the setting of semi-Segal spaces we use the following construction:
\begin{defn}
    For any semi-simplicial space $X$ we will let $X^\gd$ denote the semi-simplicial set
    obtained by forgetting the topology of $X$.
\end{defn}

\begin{lem}[{\cite[3.11(i+iv)]{St18}}]\label{lem:simplicial-str}
    If $P: X \to Y$ is a weakly unital Reedy-fibration 
    of weakly unital Reedy-fibrant semi-Segal spaces,
    then $X^\gd$ and $Y^\gd$ admit simplicial structures 
    such that they are quasicategories 
    and $P^\gd$ is an inner fibration of simplicial sets.
\end{lem}

We now study how locally (co)Cartesian fibrations behave under the functor $(\blank)^\gd$.

\begin{lem}[{\cite[Lemma 3.6]{St17}}]\label{lem:Cart-gdCart}
    Let $P: X \to Y$ be a weakly unital Reedy-fibration of weakly unital Reedy-fibrant semi-Segal spaces.
    If $e \in X_1$ is $P$-Cartesian in the sense of semi-Segal spaces (\ref{defn:sS-Cart}), 
    then $e\in X_1^\gd$ is $P^\gd$-Cartesian in the sense of quasicategories (\ref{defn:qCart}).
\end{lem}

\begin{rem}
    In the proof of \cite[Lemma 3.6]{St17} the second diagram is incorrect because
    the condition (iii) does not imply that the bottom map of this diagram is an equivalence.
    However, the underlying proof idea is still correct, and hence the proof is easily fixed as follows:
    
    The missing part of the of the proof is to show that the front square of the following cube is homotopy Cartesian. 
    This is what is referred to as ``the total square in (9)" there.
    \[
        \begin{tikzcd}[column sep = small, row sep = small]
            & X_2/e \ar[rr, "d_1"] \ar[dd, "P", near end] 
            & & X_1/x \ar[dd, "P"] \\
            X_{n+1}/e \ar[rr, crossing over, "d_n", near end] \ar[dd, "P"] \ar[ru] 
            & & X_n/x \ar[ru] & \\
            & Y_2/P(e) \ar[rr, "d_1", near start] 
            & & Y_1/P(x) &\\
            Y_{n+1}/P(e) \ar[rr, "d_n"] \ar[ru] 
            & & Y_n/P(x) \ar[from=uu, crossing over, "P", near start] \ar[ru] & 
        \end{tikzcd}
    \]
    The back square of this cube is homotopy Cartesian by the assumption that $e$ is $P$-Cartesian.
    The top square fits into a diagram as follows:
    \[
        \begin{tikzcd}
            X_{n+1}/e \ar[r, "d_n"] \ar[d] 
            & X_n/x \ar[d] \ar[r, "d_n"] 
            & X_{n-1} \ar[d] \\
            X_2/e \ar[r, "d_1"] 
            & X_1/x \ar[r, "d_1"] & X_0
        \end{tikzcd}
    \]
    Here the right-hand square and the outside rectangle are both homotopy Cartesian by the Segal condition,
    and hence the pullback pasting lemma \ref{lem:pb-pasting} implies that the left square 
    is homotopy Cartesian. This shows that the top square of the cube is homotopy Cartesian
    and a similar argument shows the same for the bottom square.
    Now the cubical pasting lemma \ref{cor:cubical-po-law} implies that the front square
    of the cube is homotopy Cartesian, which is what we claimed.
    The rest of the proof proceeds as in \cite{St17}.
\end{rem}

We now recover \cite[3.11(iii)]{St18} plus a version for locally (co)Cartesian fibrations:
\begin{cor}\label{cor:Cart-gdCart}
    Let $P: X \to Y$ be 
    a weakly unital Reedy fibration of weakly unital Reedy-fibrant semi-Segal spaces.
    Equip $X^\gd$ and $Y^\gd$ with the compatible simplicial structures from lemma \ref{lem:simplicial-str}.
    \begin{itemize}
        \item If $P$ is (co)Cartesian, then $P^\gd:X^\gd \to Y^\gd$
            is a (co)Cartesian fibration of quasicategories.
        \item If $P$ is locally (co)Cartesian, then $P^\gd:X^\gd \to Y^\gd$ 
            is a locally (co)Cartesian fibration of quasicategories.
    \end{itemize}
\end{cor}
\begin{proof}
    The first point follows directly from lemma \ref{lem:Cart-gdCart} as in \cite{St18}.
    We will prove the second point using the first.
    
    Assume that $P$ is locally Cartesian; the locally coCartesian case follows by taking opposites.
    By lemma \ref{lem:simplicial-str} the simplicial structures can be chosen such that
    $P^\gd:X^\gd \to Y^\gd$ is a simplicial map and an inner fibration.
    To see that $P^\gd$ is a locally Cartesian fibration it will suffice to check
    that for each simplicial map $\gs:\gD^1 \to Y^\gd$ the pullback 
    $P^\gd_{|\gs}: X^\gd \times_{Y^\gd} \gD^1 \to \gD^1$ is a Cartesian fibration.
    (See remark \ref{rem:locally-Cartesian}.)
    Since $\gD^1$ is level-wise discrete simplicial space $\gs$ defines a continuous map
    of semi-simplicial spaces $\gs:\gD^1 \to Y$.
    Both sides are weakly unital and $\gs$ preserves these weak units:
    The only weak units in $\gD^1$ are the two degenerate edges,
    which are send to degenerate edges in $Y$ because $\gs:\gD^1 \to Y^\gd$ is a \emph{simplicial} map.
    Every degenerate edge in $Y^\gd$ is a weak unit in $Y$ because of how the simplicial structure
    is constructed in lemma \cite[Lemma 3.11(iv)]{St17}.
    Hence $\gs: \gD^1 \to Y$ is indeed a weakly unital map of weakly unital semi-Segal spaces.
    By lemma \ref{lem:pullback-of-Cartesian} the pullback of semi-Segal spaces 
    $P_{|\gs}: X \times_Y \gD^1 \to \gD^1$ is a Cartesian fibration.
    Moreover $P_{|\gs}$ is a weakly unital Reedy fibration because $P$ was.
    Therefore, we may apply the first point of this corollary to $P_{|\gs}$ 
    and conclude that $(P_{|\gs})^\gd : (X \times_Y \gD^1)^\gd \to \gD^1$ 
    is a Cartesian fibration of simplicial sets.
    The functor $\gd$ commutes with pullbacks, so
    $(P_{|\gs})^\gd = (P^\gd)_{|\gs}: X^\gd \times_{Y^\gd} \gD^1 \to \gD^1$
    is a Cartesian fibrations of simplicial sets.
    Since we have shown this for all $\gs:\gD^1 \to Y^\gd$ it follows that $P^\gd$ is 
    a locally Cartesian fibration.
%
%
\end{proof}

We can now deduce the additivity theorem for semi-Segal spaces
from the version for quasicategories:
\begin{thm}[Main Theorem for semi-Segal spaces]\label{thm:sS-localAdd}
    Let 
    \[
        \mcB' \xrightarrow{ F } \mcB \xleftarrow{ P} \mcE
    \]
    be a diagram of weakly unital semi-Segal spaces and weakly unital maps.
    If $P$ is a level-wise fibration, a \emph{locally} Cartesian 
    and a \emph{locally} coCartesian fibration,
    then the (level-wise) pullback diagram
    \[
        \begin{tikzcd}
            \mcB' \times_\mcB \mcE \ar[r] \ar[d] & \mcE \ar[d, "P"] \\
            \mcB' \ar[r, "F"] & \mcB
        \end{tikzcd}
    \]
    realises to a homotopy pullback diagram of spaces.
\end{thm}
\begin{proof}[Proof of \ref{thm:sS-localAdd} from \ref{thm:qCatAdd}]
    The notion of weakly unital functor and locally (co)Cartesian fibration 
    are invariant under level-wise equivalences:
        For weakly unital functors this is shown in \cite[Lemma 3.9]{St18}.
        The case for locally (co)Cartesian fibrations is very similar,
        as lemma \ref{lem:weak-unit-invariant} shows that the notion of weak units
        is invariant under level-wise equivalences.
    Therefore the hypothesis of the theorem are invariant under weak equivalences and using \cite[Lemma 3.8]{St18} we may replace $\mcB$ by a Reedy-fibrant semi-Segal space
    and $F$ and $P$ by Reedy-fibrations.

    Applying the (pullback preserving) functor $(\blank)^\gd$ to the diagram we obtain a square
    \[
        \begin{tikzcd}
            (\mcB')^\gd \times_{\mcB^\gd} \mcE^\gd \ar[r] \ar[d] & \mcE^\gd \ar[d, "P^\gd"] \\
            (\mcB')^\gd \ar[r, "F^\gd"] & \mcB^\gd.
        \end{tikzcd}
    \]
    Using lemma \ref{lem:simplicial-str} we can choose simplicial structures 
    such that this is a diagram of quasicategories in which all maps are inner fibrations.
    In corollary \ref{cor:Cart-gdCart} we showed that $P^\gd: \mcE^\gd \to \mcB^\gd$ 
    is locally Cartesian and locally coCartesian, 
    so the quasicategory version of the additivity theorem (theorem \ref{thm:qCatAdd})
    applies and shows that the square realizes to a homotopy Cartesian square of spaces.

    Now the natural transformation $(\blank)^\gd \Rightarrow \Id$ 
    induces a map from the above square to the square in the statement of the proposition.
    By \cite[Lemma 3.11]{St18} this map becomes a weak equivalence in every entry 
    after geometric realisation, and so it follows that the square in the statement
    of the theorem realizes to a homotopy Cartesian square as claimed.
\end{proof}

\section{Realisation fibrations and weakly unital topological categories}\label{sec:wutCat}

In this section we prove the main theorem for realisation fibrations 
between weakly unital topological categories, which is stated below as \ref{thm:wu-locadd}.
We recall the necessary definitions, but refer the reader to \cite{ERW19} and \cite{St18}
for a more detailed introduction to non-unital topological categories.

\subsection{Weakly unital topological categories}
\begin{defn}
    A \emph{non-unital topological category} $\mcC$ is a 
    space of objects $\mcC_0$ and a space of morphisms $\mcC_1$
    together with source and target maps $s,t:\mcC_1 \to \mcC_0$ 
    and a composition map $m:\mcC_1 \times_{\mcC_0} \mcC_1 \to \mcC_1$
    such that for every composable triple 
    $(f,g,h) \in \mcC_1\times_{\mcC_0} \mcC_1 \times_{\mcC_0} \mcC_1$ we have
    \begin{align*}
        s(m(f,g)) &= s(g) & t(m(f,g)) &= t(g) & m(f,m(g,h)) &= m(m(f,g),h).
    \end{align*}
    A \emph{non-unital functor} $F:\mcC \to \mcD$ is a pair of continuous maps 
    $F_0:\mcC_0 \to \mcD_0$ and $F_1:\mcC_1 \to \mcD_1$ commuting with $s$, $t$, and $m$.
\end{defn}

\begin{defn}
    The \emph{classifying space} $B\mcC$ of $\mcC$ is the (fat) geometric realisation 
    of its semi-simplicial \emph{nerve} $N\mcC$.
\end{defn}

\begin{defn}
    For two objects $x,y \in \mcC_0$ we define the $\hom$-space as the pullback
    \[
        \mcC(x,y) := \{x\} \times_{\mcC_0} \mcC_1 \times_{\mcC_0} \{y\}
    \]
    and we write $f:x \to y$ for $f \in \mcC(x, y)$.
\end{defn}

%

\begin{defn}\label{defn:wut-cat}
    We say that an endomorphism $u \in \mcC(x,x)$ is a \emph{weak unit} 
    if for all objects $w, z \in \mcC_0$ both pre- and post-composition with $u$
    \[
        f_*: \mcC(w,x) \to \mcC(w,x) \qand f^*: \mcC(x,z) \to \mcC(x,z)
    \]
    are homotopic to the identity.
    
    A non-unital topological category $\mcC$ is \emph{weakly unital} 
    if for every object $x \in \mcC$ there is a weak unit $u \in \mcC(x, x)$.
    A \emph{weakly unital functor} is a non-unital functor that sends weak units to weak units.
\end{defn}

\begin{rem}
    In Steimle's paper \cite{St18} weakly unital topological categories are not defined 
    in terms of weak units but rather in terms of equivalences. 
    A morphism $f:x \to y$ is an equivalence if and only if pre- and post-composition with it
    induce equivalence $f_*:\mcC(w,x) \simeq \mcC(w,y)$ and $f^*:\mcC(y,z) \simeq \mcC(x,z)$.
    A weakly unital topological category in his sense is a non-unital topological category
    where every object is the target of an equivalence.
    Since weak units are in particular equivalences, it follows that every weakly unital 
    topological category in the sense of definition \ref{defn:wut-cat} is also weakly unital
    in Steimle's sense.
    
    The converse also holds: if an object $y$ is the target of some equivalence $f:x \to y$,
    then one can use the fact that $f^*:\mcC(y,y) \to \mcC(x,y)$ is surjective on path-components
    to find $u:y \to y$ with $f \circ u \simeq f$. From this it follows that $u$ itself
    is an equivalence and further that $u \circ u \simeq u$. So $u_*: \mcC(w, y) \to \mcC(w, y)$
    is a homotopy automorphism satisfying $u_* \circ u_* \simeq u_*$ 
    -- it has to be homotopic to the identity.
    For the same reason $u^*$ is homotopic to the identity and so $u$ is a weak unit. 
    Therefore the notion of a weakly unital topological category in definition
    \ref{defn:wut-cat} agrees with the one from \cite{St18}.
    A similar argument shows that the notions of weakly unital functor agree.
\end{rem}

\begin{defn}
    A functor $F:\mcC \to \mcD$ of non-unital topological categories is a \emph{local fibration}
    if the continuous maps $F_0:\mcC_0 \to \mcD_0$ and
    $((s,t),F_1):\mcC_1 \to (\mcC_0)^2 \times_{(\mcD_0)^2} \mcD_1$ are Serre fibrations.

    A non-unital topological category $\mcC$ is locally fibrant if the functor $\mcC \to *$
    is a local fibration. 
    Explicitly, this means that the ``source-target map'' 
    $(s,t):\mcC_1 \to (\mcC_0)^2$ is a Serre fibration.
\end{defn}

\begin{rem}
    If a non-unital topological category $\mcC$ is fibrant, then its nerve $N\mcC$ is a semi-Segal space:
    by definition of the nerve we have $N_k\mcC \cong N_1\mcC \times_{N_0\mcC} \dots \times_{N_0\mcC} N_1\mcC$,
    and for $\mcC$ fibrant the relevant maps in these pullbacks are fibrations, making them homotopy pullbacks.
    Moreover, in this case the construction $N_n\mcC/\gs$ from definition \ref{defn:X(A)} 
    is equivalent to the (strict) fiber of $N_n\mcC \to N_k\mcC$ at $\gs \in N_k\mcC$.
    This in turn is homeomorphic to the fiber of $N_{n-k}\mcC \to N_0\mcC$ over $x$,
    where $x = \gs(0)$ is the initial vertex of $\gs$.
    
    Note that in this case the above definition of a weak unit in $\mcC$ coincides 
    with the definition of a weak unit in $N\mcC$ as given in \ref{defn:weak-unit}.
\end{rem}


\begin{defn}\label{defn:Cart}
    Let $P:\mcE \to \mcB$ be a weakly unital functor: 
    \begin{itemize}
        \item A morphism $f:e \to e'$ in $\mcE$ is \emph{$P$-Cartesian}
            if 
            the following diagram is homotopy Cartesian for all $t \in \mcE_0$:
            \[
                \begin{tikzcd}[column sep = large]
                    \mcE(t, e) \ar[r, "f \circ \blank"] \ar[d, "P_{(t,e)}"] 
                    & \mcE(t, e') \ar[d, "P_{(t,e')}"] \\
                    \mcB(P(t), P(e))\ar[r, "P(f) \circ \blank"] 
                    & \mcB(P(t), P(e')) 
                \end{tikzcd}
            \]
        \item A morphism $f:e \to e'$ is \emph{locally $P$-Cartesian}, 
            if for all $t \in \mcE$ with $P(t) = P(e)$ and all weak units $u:P(t) \to P(e)$,
            the above square is homotopy Cartesian at $u$.
            Spelling out definition \ref{defn:locCartSqr}
            this means that the induced map on homotopy fibers 
            \[
                (f\circ\blank): 
                \hofib_{u}\left(\mcE(t,e) \xrightarrow{P_{(t,e)}} \mcB(P(t),P(e)) \right)
                \xrightarrow{\ \simeq\ } 
                \hofib_{P(f) \circ u}\left(\mcE(t,e') \xrightarrow{P_{(t,e')}} \mcB(P(t), P(e'))\right)
            \] 
            is a weak equivalence.
        \item We say that $P$ is \emph{(locally) Cartesian} if for $e' \in \mcE$, $b \in \mcB$,
            and $g:b \to P(e')$ there is a (locally) $P$-Cartesian $f:e \to e'$ such that 
            $P(e) = b$ and $P(f) = g$.
        \item We say that $P$ is \emph{(locally) coCartesian} if $P^{op}:\mcE^{op} \to \mcB^{op}$ 
            is (locally) Cartesian.
    \end{itemize}
\end{defn}

\begin{rem}
    Note that any two weak units $u,u'\in \mcB(P(e), P(e))$ must lie in the same path-component
    because we have paths $u \simeq u'^*(u) = u \circ u' = u_*(u') \simeq u'$ 
    coming from the homotopies $\id \simeq u'^*$ and $\id \simeq u_*$. 
    Hence, when checking whether a morphism $f:e \to e'$
    is locally $P$-Cartesian it suffices to do so at a single weak unit $u:P(t) = P(e) \to P(e)$.
\end{rem}

We now have all the definitions at hand to state Theorem \ref{thm:main}.
\begin{thm}[Main Theorem for weakly unital topological categories]\label{thm:wu-locadd}
    Let $P: \mcE \to \mcB$ be a weakly unital functor of weakly unital topological categories
    such that $\mcB$ is fibrant, $P$ is a local fibration, $P$ is \emph{locally} Cartesian,
    and $P$ is \emph{locally} coCartesian. Then $P$ is a realisation fibration.

    Hence, for every weakly unital functor $F:\mcC \to \mcB$,
    the following square is homotopy Cartesian:
    \[
        \begin{tikzcd}
            B(\mcC \times_{\mcB} \mcE) \ar[r] \ar[d] & B\mcE \ar[d, "P"] \\
            B\mcC \ar[r, "F"] & B\mcB .
        \end{tikzcd}
    \]
\end{thm}

Before we prove this from Theorem \ref{thm:sS-localAdd}, we first need to translate 
from weakly unital categories to semi-Segal spaces.
The (not locally) Cartesian part of the following lemma is \cite[Remark 2.8(ii+iii)]{St18}.

\begin{lem}\label{lem:locCart-wu-sS}
    Consider a weakly unital functor $P:\mcE \to \mcB$ of weakly unital, locally fibrant categories,
    such that $P$ is a local fibration. 
    Then $P:\mcE \to \mcB$ is locally Cartesian if and only if $NP:N\mcE \to N\mcB$ is locally Cartesian
    in the sense of semi-Segal spaces.
\end{lem}
\begin{proof}
    Fix some morphism $f:e \to e'$ in $\mcE$. We would like to show that $f$ is locally $P$-Cartesian,
    if and only if the corresponding edge $f \in N_1\mcE$ is locally $NP$-Cartesian.
    We will do this by applying lemma \ref{lem:local-Cartesian-via-fibers} to the following diagram.
    \[
        \begin{tikzcd}
            N_2\mcE/f \ar[r, "d_1"] \ar[d, "P_2"] \ar[dr, phantom, "*"] & 
            N_1\mcE/e' \ar[r, "d_1"] \ar[d, "P_1"] &
            N_0\mcE \ar[d, "P_0"] \\
            N_2\mcB/P(f) \ar[r, "d_1"] & 
            N_1\mcB/P(e') \ar[r, "d_1"] & 
            N_0\mcB
        \end{tikzcd}
    \]
    By definition, $f$ is $NP$-Cartesian if and only if the square $(*)$ is homotopy Cartesian
    at the path-component of $N_2\mcB/P(f)$ that contains those $2$-simplices whose 
    first morphism is a weak unit. (Note that in general this component will also contain $2$-simplices
    whose first edge is not a weak unit.)
    To be concrete, let $u:P(e) \to P(e)$ be a weak unit
    and let $\gs \in N_2\mcB/P(f)$ be the $2$-simplex defined by $u$ and $P(f)$.
    Lemma \ref{lem:local-Cartesian-via-fibers} says that we can check whether $(*)$ is
    homotopy Cartesian at $\gs$ by passing to homotopy fibers with respect to the map
    to $N_0\mcE$ and $N_0\mcB$. The image of $\gs$ in $N_0\mcB$ is $P(e)$, and we let
    $t \in N_0\mcE$ with $P(t) = P(e)$ be a potential lift.
    
    The homotopy fiber of the map $d_1:N_1\mcE/e' \to N_0\mcE$ at $t$ is $t\backslash N_1\mcE /e'$,
    which is equivalent to the strict fiber $\mcE(t, e') = N_1 \mcE \times_{(N_0\mcE)^2} \{(t,e')\}$
    because we assumed that $\mcE$ is locally fibrant.
    Similarly, the homotopy fiber of $d_1 \circ d_1:N_2 \mcE/f \to N_0\mcE$ at $t$
    is equivalent to $\mcE(t, e)$.
    Hence the diagram obtained from $(*)$ by passing to homotopy fibers is equivalent to:
    \[
        \begin{tikzcd}
            \mcE(t, e) \ar[r, "(f \circ \blank)"] \ar[d] \ar[dr, phantom, "\dag_t"] &
            \mcE(t, e') \ar[d] \\ 
            \mcB(P(e), P(e)) \ar[r, "d_1"] & \mcB(P(e), P(e'))
        \end{tikzcd}
    \]
    Lemma \ref{lem:local-Cartesian-via-fibers} says that $(*)$ is homotopy Cartesian at $\gs$
    if and only if $(\dag_t)$ is homotopy cartesian at the weak unit $u:P(e) \to P(e)$ for all $t$.
    In other words, $f$ is locally $NP$-Cartesian, if and only if it is locally $P$-Cartesian.
    
    Since we have now shown that the notions of locally Cartesian edges agree,
    it follows that the subspace $N_1\mcC^{NP\rm -lCart} \subset N_1\mcB$ 
    in the sense of definition \ref{defn:sS-Cart} consists exactly of the $P$-Cartesian edges.
    The semi-simplicial map $NP$ is locally Cartesian if and only if the map 
    $(d_1,P): N_1\mcE^{NP\rm -lCart} \to N_0\mcE \times_{N_0\mcB} N_1\mcB$
    is surjective on path-components, and the functor $P$ is locally Cartesian 
    if and only if this map is surjective on-the-nose.
    But $(d_1,P)$ is a Serre fibration because $P$ is a local fibration,
    so the two are equivalent.
\end{proof}

\begin{proof}[Proof of theorem \ref{thm:wu-locadd}]
    Let $P: \mcE \to \mcB$ be a weakly unital functor of weakly unital topological categories
    such that $\mcB$ is fibrant, $P$ is a fibration, $P$ is \emph{locally} Cartesian,
    and $P$ is \emph{locally} coCartesian. 
    Then we need to show that $P$ is a realisation fibration.
    In other words, for every weakly unital functor $F:\mcC \to \mcB$ 
    we need to show that the following square is homotopy Cartesian:
    \[
        \begin{tikzcd}
            B(\mcC \times_{\mcB} \mcE) \ar[r] \ar[d] & B\mcE \ar[d, "P"] \\
            B\mcC \ar[r, "F"] & B\mcB.
        \end{tikzcd}
    \]
    We do this by applying theorem \ref{thm:sS-localAdd}, 
    the version of the additivity theorem for semi-Segal spaces,
    to the diagram of semi-Segal spaces
    \[
        N\mcC \xrightarrow{\ F\ } N\mcB \xleftarrow{\ P\ } N\mcE.
    \]
    This is a diagram of weakly unital maps between semi-Segal spaces 
    by \cite[Remark 2.10]{St18} and $NP$ is locally Cartesian and locally coCartesian
    by lemma \ref{lem:locCart-wu-sS}.
    Moreover, $NP$ is a level-wise fibration because $P$ is a local fibration.
    The conclusion of \ref{thm:sS-localAdd} indeed proves our claim
    since the classifying space $B\mcC$ is exactly defined as 
    the geometric realisation of $N\mcC$.
\end{proof}

\section{An application: the cospan category}\label{sec:Cosp}

The classifying space of various cobordism categories have been computed via geometric 
methods, most notably in \cite{GMTW06}.
These cobordism categories are, in some sense, categories of cospans of manifolds.
In this section we apply Theorem \ref{thm:main} to study the category 
of cospans of finite sets, which one might think of as a 
combinatorial bordism category.

\subsection{The strict model for the cospan category}
In order to study the classifying space of $\Csp$ we need a strictly associative model
for the cospan category. To achieve this, we need to be extremely careful
about how we construct pushouts.
We therefore establish some notation.

\begin{notation}
    For any integer $n \ge 0$ we let $\ul{n}$ denote the finite set
    \[
        \ul{n} := \{1, \dots, n\} \subset \IZ.
    \]
    Given an equivalence relation $R$ on $\ul{n}$ we write $xRy$ to denote that $x$ is equivalent to $y$
    under $R$.
    The equivalence class of $i \in \ul{n}$ will be denoted by
    \[
        [i]_R := \{ j \in \ul{n} \;|\; i R j\}.
    \]
    The quotient of $\ul{n}$ by $R$ is the set of equivalence classes
    \[
        \ul{n}/R := \{ [i]_R \;|\; i \in \ul{n}\} \subset \mathfrak{P}(\ul{n}),
    \]
    which is technically a certain subset of the power-set of $\ul{n}$.
\end{notation}

\begin{rem}
    Note that it is a property for a set $X$ to be of the form $\ul{n}/R$.
    Concretely, if for two numbers $n,m \ge 0$ and two relations $R$ and $S$
    the two sets $\ul{n}/R$ and $\ul{m}/S$ are equal, then $n=m$ and $R=S$.
    Of course we can still have $\ul{n}/R \cong \ul{m}/S$ for $n \neq m$.
\end{rem}

\begin{defn}
    For two finite sets $A$ and $B$ we let $\Csp(A,B)$ 
    be the groupoid with objects cospans $(A \to W \leftarrow B)$
    such that $W = \ul{n}/R$ 
    for some $n\ge 0$ and some equivalence relation $R$ on $\ul{n}$.
    A morphism in this groupoid is a bijection $\gs$ making the following diagram commute:
    \[
        \begin{tikzcd}[row sep = tiny]
            & W \ar[dd, "\gs"] & \\
            A \ar[ru, "i"]\ar[rd, "i'"'] && B \ar[lu, "j"']\ar[ld, "j'"] \\
                            & W' & 
        \end{tikzcd}
    \]
\end{defn}

Next, we choose an explicit representative of the pushout of two cospans:
\begin{defn}
    For three finite sets $A,B,C$ we define the functor
    \[
        \mu: \Csp(B,C) \times \Csp(A,B) 
        \longrightarrow \Csp(A,C)
    \]
    by sending $(B \to \ul{m}/S \leftarrow C)$ 
    and $(A \to \ul{n}/R \leftarrow B)$ to the 
    cospan $(A \to (\ul{n+m})/Q \leftarrow C)$.
    Here the equivalence relation $Q$ on $\ul{n+m}$ is generated by the following requirements:
    \begin{align*}
        \forall x,y \in \ul{n} :x R y &\Rightarrow x Q y    
                                   & \forall z,w \in \ul{m}: z S w &\Rightarrow (z+n) Q (w+n) 
                                   & \forall b \in B: j(b) Q l(b) .
    \end{align*}
    In other words, $Q$ is the \emph{unique} equivalence relation on $\ul{n+m}$ 
    such that the diamond in the following diagram is well defined and a pushout square:
    \[
        \begin{tikzcd}[row sep = tiny]
            & & \ul{n+m}/Q & & \\
            & \ul{n}/R \ar[ru] && \ul{m}/S \ar[lu, "\blank+n"']  & \\
            A \ar[ru, "i"] && B \ar[lu, "j"'] \ar[ru, "k"] && C \ar[lu, "l"'] .
        \end{tikzcd}
    \]
    On morphisms $\mu$ is defined via the universal property of the pushout.
\end{defn}

\begin{rem}
    It is straightforward to check that $\mu$ is strictly associative.
    The composite of the two cospans above with a third cospan $(C \to \ul{l}/U \leftarrow D)$
    will be of the form $(A \to (\ul{n+m+l})/T \leftarrow D)$ for some equivalence relation
    $T$ on $\ul{n+m+l}$. 
    This equivalence relation is uniquely determined by a universal property
    and does not depend on which order we composed the cospans in.
\end{rem}

\begin{rem}\label{lem:Csp-str}
    Usually the category of cospans is defined as a bicategory $\mrm{biCsp}$, 
    see for example \cite{Ben67}.
    The above defines a non-unital strict $2$-category $\Csp^{\mrm{str}}$ with objects
    finite sets and $\hom$-categories the groupoids $\Csp(A,B)$.
    Every bicategory can be strictified and $\Csp^{\mrm{str}}$ is a specific
    choice of a (non-unital) strictification of $\mrm{biCsp}$.
    Indeed the canonical map $\Csp^{\mrm{str}} \to \mrm{biCsp}$ is a biequivalence,
    in the sense that it is essentially surjective and an equivalence on $\hom$-categories.
    To see this, observe that every finite set is in bijection with a set of the form 
    $\ul{n}/R$, and 
    therefore every cospan $(A \to W \leftarrow B)$ is isomorphic to a cospan
    of the form $(A \to \ul{n}/R \leftarrow B)$. 
\end{rem}

\begin{defn}\label{defn:topCosp}
    The non-unital topologically enriched category $\Csp$ has as objects natural numbers 
    $n \ge 0$ and as $\hom$-spaces\footnote{
        Here $B$ is defined by taking the simplicial nerve and then the standard geometric realisation.
        In particular, $B$ commutes with Cartesian products. 
        This would not be the case if we took the ``fat geometric realisation''.
    }
    \[
        \hom_{\Csp}(a,b) := B\Csp(\ul{a}, \ul{b}).
    \]
    Composition is induced by the functor $\mu$.

    We define $\hCsp$ as the ordinary category that is obtained from $\Csp$ by identifying
    two cospans whenever there is an isomorphism between them.
    We let $P: \Csp \to \hCsp$ denote the quotient functor 
    that sends a cospan $(\ul{a} \to \ul{n}/R \leftarrow \ul{b})$ to its isomorphism class.
\end{defn}

\begin{lem}\label{lem:hom-Csp}
    The above defines a weakly unital topological category $\Csp$.
    Moreover, the $\hom$-spaces in $\Csp$ decompose as
    \[
        \hom_{\Csp}(A, B) 
        \simeq \coprod_{[A \to W \leftarrow B] \in \hCsp(A,B)} B\gS_k
    \]
    where $\gS_k$ is the symmetric group of order
    $k = |W \setminus \mi{Im}(A \amalg B \to W)|$.
\end{lem}
\begin{proof}
    $\Csp$ defines a non-unital topologically enriched category:
    we observed earlier that $\mu$ is strictly associative, 
    so $B\mu$ is strictly associative, too.
    It has weak units of the form $[A = A = A]$.
    
    Fixing two finite sets $A$ and $B$, we would like to understand the space $\hom_{\Csp}(A,B)$,
    or equivalently, the groupoid $\Csp(A,B)$.
    It decomposes as a disjoint union of its isomorphism classes,
    which by definition are the elements of $\hCsp(A,B)$.
    All that remains to show that the group of automorphisms of a cospan $(A \to W \leftarrow B)$
    is indeed $\gS_k$ for $k = |W \setminus \mi{Im}(A\amalg B \to W)|$.
    
    Let $\ga$ be an automorphism of $(A \to W \leftarrow B)$.
    Then $\ga: W \to W$ is a bijection, which has to fix both the image of $A \to W$
    and the image of $B \to W$.
    So 
    \[
        \Aut(A \to W \leftarrow B) \cong \Aut(W \setminus \mi{Im}(A \amalg B \to W)) \cong \gS_k
    \]
    and the claim follows.
\end{proof}

\subsection{Proof of theorem \ref{thm:Cosp-fib}}
According to lemma \ref{lem:hom-Csp} the difference between $\Csp$ and $\hCsp$ 
lies in the fact that a cospan $(A \to W \leftarrow B)$ in $\Csp$ has non-trivial 
automorphisms defined by permuting points in $W \setminus \mrm{Im}(A \amalg B \to W)$.
We now consider ``reduced cospans'', for which this set is trivial:
\begin{defn}
    For a cospan $(A \to W \leftarrow B)$ we define its \emph{reduced} and its \emph{closed} part as
    \[
        r(W) := \mi{Im}(A \amalg B \to W) 
        \qand 
        c(W) := W \setminus r(W).
    \]
    Here we suppress the dependence of $r(W)$ and $c(W)$ on the maps $A,B \to W$ in the notation.
    We call the cospan reduced if $W=r(W)$ and closed if $W=c(W)$.
\end{defn}

\begin{defn}
    The reduced cospan category $\rCsp$ has as objects finite sets $\ul{a}$
    as morphisms isomorphism classes of reduced cospans. 
    Composition is defined by
    \[
        [\ul{b} \to V \leftarrow \ul{c}] \circ [\ul{a} \to W \leftarrow \ul{b}] 
        := [\ul{a} \to r(W\cup_B V) \leftarrow \ul{c}].
    \]
    We let $R$ denote the reduction functor
    \[
        R:\hCsp \to \rCsp, 
        \qquad 
        [\ul{a} \to W \leftarrow \ul{b}] \mapsto [\ul{a} \to r(W) \leftarrow \ul{b}].
    \]
\end{defn}

\begin{rem}
    To show that the composition in $\rCsp$ is associative, 
    we need to check that for any three composable cospans:
    \[
        [\ul{a} \to r(r(W \cup_B V) \cup_C U) \leftarrow \ul{d}]
        =
        [\ul{a} \to r(W \cup_B r(V \cup_C U)) \leftarrow \ul{d}].
    \]
    The two sets $r(r(W \cup_B V) \cup_C U)$ and $r(W \cup_B r(V \cup_C U))$ can both  
    be thought of as subsets of the double-pushout $W \cup_B V \cup_C U$.
    In fact, both are the image of $\ul{a} \amalg \ul{d} \to W \cup_B V \cup_C U$, so they agree.
\end{rem}

\begin{rem}\label{rem:red-not-sub}
    Note that the composition of two reduced cospans is not always reduced:
    \[
        (\ul{1} \xrightarrow{1} \ul{2} \xleftarrow{2} \ul{1}) \circ (\ul{0} \to \ul{1} \leftarrow \ul{1})
        = (\ul{0} \to \ul{3}/(1\sim 2) \xleftarrow{3} \ul{1}) 
        \cong (\ul{0} \to \ul{2} \xleftarrow{2} \ul{1}) .
    \]
    (Recall that the order of composition in the cospan category is counter-intuitive.)
    This shows that it is essential to define the composition in $\rCsp$ as $r(W \cup_B V)$,
    since the pushout $W \cup_B V$ might not be reduced.
    It also shows that $\rCsp$ is not a subcategory of $\hCsp$, but rather a quotient category:
    
    The reduction functor is a bijection on objects and surjective on morphisms.
    We can hence define $\rCsp$ as the quotient of $\Csp$ by the relation 
    that identifies two cospans if their reduced parts are isomorphic.
\end{rem}

We now wish to compute the homotopy fiber of the quotient functor $R:\hCsp \to \rCsp$.
To apply theorem \ref{thm:wu-locadd} to the functor $R$ we need to check
that it is locally Cartesian and coCartesian:

\begin{lem}\label{lem:RlocCart}
    The reduction functor 
    \[
        R:\hCsp \to \rCsp
    \]
    is locally Cartesian and locally coCartesian.
    However, it is neither Cartesian nor coCartesian.
\end{lem}
\begin{proof}
    Let us first look at the fiber of the map
    \[
        R:B\hCsp(A,B) \to B\rCsp(A,B) 
    \]
    at some reduced cospan $[A \to U \leftarrow B]$.
    Every cospan in the fiber is of the form 
    \[
        [A \to W \leftarrow B] 
        = [A \to U \leftarrow B] \amalg [\emptyset \to c(W) \leftarrow \emptyset].
    \]
    Hence the fiber is in bijection with $\IN$ by counting the elements 
    of $c(W) = W \setminus \mi{Im}(A \amalg B \to W)$.

    Consider a cospan $[A \to W \leftarrow B] \in \hCsp$; we will show that it is 
    locally $R$-Cartesian precisely if it is reduced.
    By definition \ref{defn:locCart} the cospan $W$ is locally $R$-Cartesian 
    precisely if the top map in the following diagram is a bijection:
    \[
        \begin{tikzcd}[column sep = 4pc]
            R^{-1}([A=A=A]) \ar[r, "{(\blank \cup_A W)}"] \ar[d, "\cong"]
            &  R^{-1}([A \to r(W) \leftarrow B]) \ar[d, "\cong"] \\
            \IN \ar[r, "(\blank + |c(W)|)"] & \IN
        \end{tikzcd}
    \]
    This diagram commutes,
    and hence $[A \to W \leftarrow B]$ is locally $R$-Cartesian precisely if $|c(W)| = 0$,
    i.e.\ if it is reduced.

    The proof that $[A \to W \leftarrow B]$ is locally $R$-coCartesian precisely if it is reduced,
    is completely analogous.
    In fact, all variants of the cospans categories are isomorphic
    to their opposite via the functor that sends $(A\to W \leftarrow B)$
    to $(B \to W \leftarrow A)$. 
    Since $R$ commutes with this involution, the claim that a morphism 
    is locally $R$-coCartesian if and only if it is reduced follows formally.
    
    To see that $R$ is indeed locally Cartesian and locally coCartesian,
    we observe that any morphism in $\rCsp$ has a reduced lift to $\Csp$.

    Finally, we show that $R$ is neither a Cartesian nor a coCartesian fibration.
    Say $R$ was Cartesian, then by \cite[Proposition 2.4.2.8]{LurHTT} 
    the composite of two locally $R$-Cartesian morphisms ought to be locally $R$-Cartesian.
    But as we saw in remark \ref{rem:red-not-sub} the composite of two reduced cospans 
    is not always reduced, leading to a contradiction.
\end{proof}

Similarly, we have:
\begin{lem}\label{lem:R2locCart}
    The reduction functor 
    \[
        R \circ P: \Csp \to \rCsp
    \]
    is locally Cartesian and locally coCartesian.
\end{lem}
\begin{proof}
    Let $\mi{Fin}^\simeq$ denote the groupoid of finite sets and bijections.
    The fiber of 
    \[
        R \circ P: \Csp(A,B) \to \rCsp(A,B)
    \]
    at some cospan $[A \to W \leftarrow B]$ is equivalent to $B\mi{Fin}^\simeq$.
    The equivalence is the restriction of the functor
    \[
        \Csp(A,B) \to \mi{Fin}^\simeq \quad (A \to W \leftarrow B) 
        \mapsto 
        c(W) = W \setminus \mi{Im}(A \amalg B \to W)
    \]
    to the fiber.
    Just like in lemma \ref{lem:RlocCart} we see that for reduced cospans in 
    $\Csp(A,B)$ the induced maps on the fiber over weak units is the identity 
    on $\mi{Fin}^\simeq$.
    Therefore reduced cospans satisfy the condition in point 2 of definition \ref{defn:Cart},
    which means that they are locally $(R\circ P)$-Cartesian in the sense of weakly unital
    topological categories.
    
    The rest of the proof follows like in lemma \ref{lem:RlocCart}.
\end{proof}
        
With the two previous lemmas at hand we can now apply theorem \ref{thm:wu-locadd}:
\begin{proof}[Proof of Theorem~\ref{thm:Cosp-fib}]
    Let $*$ and $\IN$ denote the unital categories with one object and endomorphisms
    $\{\id_*\}$ and $\IN$, respectively. 
    There is a pullback diagram of categories
    \[
        \begin{tikzcd}
            \IN \ar[r] \ar[d] & \hCsp \ar[d, "R"] \\
            * \ar[r] & \rCsp
        \end{tikzcd}
    \]
    where the bottom functor sends $*$ to the object $\emptyset \in \rCsp$.
    Theorem \ref{thm:wu-locadd} applies to this diagram:
    all categories involved are unital and discrete, hence weakly unital and fibrant,
    and lemma \ref{lem:RlocCart} verifies the only non-trivial condition.
    This shows that the following is a homotopy fiber sequence:
    \[
        S^1 \simeq B\IN \longrightarrow B\hCsp \xrightarrow{ \ R \ } B\rCsp.
    \]

    Similarly, we obtain the fiber sequence for $\Csp$:
    let $\mc{F}$ denote the full (topological) subcategory of $\Csp$ 
    on the object $\emptyset$. 
    Then there is a pullback diagram of topologically enriched
    non-unital categories:
    \[
        \begin{tikzcd}
            \mc{F} \ar[r] \ar[d] & \Csp \ar[d, "R \circ P"] \\
            * \ar[r] & \rCsp.
        \end{tikzcd}
    \]
    Theorem \ref{thm:wu-locadd} applies to this diagram: lemma \ref{lem:R2locCart} 
    checks that $R \circ P$ is locally Cartesian and locally coCartesian;
    moreover, $R \circ P$ is trivially a local fibration since $\rCsp$ 
    and the space of objects of $\Csp$ are discrete.

    To obtain the desired homotopy fiber sequence recall from the proof of \ref{lem:R2locCart}
    that $\mc{F}$ 
    is the monoidal groupoid $\mi{Fin}^\simeq$ thought of as a bicategory with one object.
    Its classifying space is hence the classifying space of the topological monoid
    \[
        B \mi{Fin}^\simeq \simeq \coprod_{k \ge 0} B\gS_k.
    \]
    By the Barratt-Priddy-Quillen theorem we have that
    \[
        \gO B(B\mi{Fin}^\simeq) \simeq Q(S^0)
    \]
    and therefore $B\mcF \simeq B(B\mi{Fin}^\simeq, \amalg) \simeq Q(S^1)$.
    Under this identification the map $B(B\mi{Fin}^\simeq, \amalg) \to B\IN$
    is the infinite loop space map $QS^1 \to S^1$ which induces an isomorphism on $\pi_1$.
    Its fiber is the universal covering of $QS^1$ 
    and this proves the final claim in Theorem \ref{thm:Cosp-fib}.
\end{proof}

\begin{rem}
    We expect that the diagram in Theorem \ref{thm:Cosp-fib} is a diagram of infinite loop spaces.
    For the bottom row this can be shown using standard infinite loop space machinery,
    because all the categories involved are symmetric monoidal.
    Establishing this fact for the top is beyond the scope of this paper, 
    but will be addressed as part of future work.
\end{rem}


\subsection{Comparison to the cospan category of Raptis and Steimle}
We now apply our techniques to the cobordism category $\Cob(\mcC)$
defined  in \cite{RS19}.
In the special case where $\mcC$ is the category of finite sets
with injections as cofibrations their result says:
\[
    \gO B\Cob(\Fin) \simeq K(\Fin) \simeq Q(S^0).
\]
In fact, this cobordism category of finite sets is equivalent to a certain subcategory
of $\Csp$:
\begin{defn}
    We let $\iCsp \subset \Csp$ be the subcategory that contains all objects,
    but only those morphisms $(A \to W \leftarrow B)$ where the map $A \to W$ is injective.
    Let $\riCsp \subset \rCsp$ denote the subcategory defined 
    by the same condition.
\end{defn}

\begin{rem}
    In \cite{RS19} $\Cob(\mcC)$ is not defined as a topological category, but rather as a simplicial space $\Cob(\mcC)_\cd$.
    When the weak equivalences in the Waldhausen category $\mcC$ are the isomorphisms,
    the simplicial space $\Cob(\mcC)_\cd$ satisfies the Segal-condition and can hence be thought of as an $(\infty,1)$-category.
    In the case at hand one can define a simplicial map 
    $N_\cd\iCsp \to \Cob(\Fin)_\cd$ by sending an $n$-simplex 
    $(\ul{a}_0 \to W_1 \leftarrow \ul{a}_1 \to \dots \leftarrow \ul{a}_n) \in N_n\iCsp$
    to the functor $F:\mathrm{tw}[n] \to \Fin$ defined by 
    $F_{ij} := (W_{i+1} \cup_{\ul{a}_{i+1}} \dots \cup_{\ul{a}_{j-1}} W_j)$.
    (Here we are using the notation of \cite[section 2]{RS19}.)
    This map $N_\cd\iCsp \to \Cob(\Fin)_\cd$ is a Dwyer-Kan equivalence,
    and hence in particular induces an equivalence on geometric realisations.
    We may therefore work with $\iCsp$ instead $\Cob(\Fin)_\cd$ in what follows.
\end{rem}

We can now give a new proof of the main result of \cite{RS19}
in the example $\mcC = \Fin$:

\begin{proof}[Proof of Corollary \ref{cor:RS-Fin}]
    Consider the diagram of weakly unital topological categories:
    \[
        \begin{tikzcd}
            \iCsp \ar[d, "I"] \ar[r, "Q"] & \riCsp \ar[d, "{I'}"]\\
            \Csp \ar[r, "R \circ P"] & \rCsp.
        \end{tikzcd}
    \]
    It is a pullback diagram because a cospan $(A \to W \leftarrow B) \in \Csp$
    lies in the subcategory $\iCsp$, precisely if the reduced cospan 
    $(A \to r(W) \leftarrow B) \in \rCsp$ lies in the subcategory $\riCsp$.
     
    Using Theorem \ref{thm:main}
    we showed in \ref{lem:R2locCart} that $R\circ P$ is a realisation fibration.
    Therefore the above square becomes a homotopy pullback square after geometric realisation.
    The inclusion $\mcF \to \Csp$ of the full subcategory on $\emptyset$
    factors through the subcategory $\iCsp$
    and we hence have the following commutative diagram:
    \[
        \begin{tikzcd}
            B\mcF \ar[r, "F"] \ar[d, equal] & B\iCsp \ar[d, "I"] \ar[r, "Q"] 
                                             & B\riCsp \ar[d, "{I'}"]\\
            B\mcF \ar[r] & B\Csp \ar[r, "R \circ P"] 
                          & B\rCsp
        \end{tikzcd}
    \]
    Both rows are fiber sequences and we have that $B\mcF \simeq Q(S^0)$
    as in the proof of Theorem \ref{thm:Cosp-fib}.
    
    It remains to check that $B\riCsp$ is contractible.
    First, note that $\riCsp$ is an ordinary category: its $\hom$-spaces are discrete.
    In fact, it is equivalent to the category of finite sets and partially defined maps.
    For our purposes it suffices to note that the object $\emptyset = \ul{0}$ is terminal:
    any reduced and injective cospan $(A \to W \leftarrow \ul{0})$ is isomorphic 
    to $(A = A \leftarrow \ul{0})$.
    This proves the final claim of the corollary as the classifying space 
    of any category with a terminal object is contractible.
\end{proof}
 
\begin{rem}
    Note that the proof of corollary C remains almost unchanged if we replace $\iCsp$
    by the subcategory $\iiCsp \subset \iCsp$ where both legs are required to be injective.
    In particular the object $\emptyset$ is still terminal in $\riiCsp$.
    Hence we also have that $B\iiCsp \simeq B\iCsp \simeq Q(S^1)$.
\end{rem}

\bibliography{literature}{}
\bibliographystyle{alpha}

\bigskip
{\footnotesize
Jan Steinebrunner, \textsc{DPMMS, Centre for Mathematical Sciences, Wilberforce Road, Cambridge CB3 0WB,
UK} \par\nopagebreak
  \textit{E-mail address}:  \texttt{js2675@cam.ac.uk}.
}

\end{document}